\documentclass[12pt]{article}

\usepackage{amsmath,amsthm,amsfonts,txfonts,amssymb,amscd,epsf,epsfig,psfrag,enumerate}
\usepackage[mathscr]{eucal}
\usepackage{yhmath}
\usepackage{graphicx,epstopdf}
\usepackage{float}
\usepackage{wasysym}

\title{Finite height lamination spaces for surfaces}

\author{Ulrich Oertel}

\date{April, 2014}

\newenvironment{tightenum}{
\begin{enumerate}[(i)]
  \setlength{\itemsep}{1pt}
  \setlength{\parskip}{0pt}
  \setlength{\parsep}{0pt}
}{\end{enumerate}}
\newtheorem{thm}{Theorem}[section] \newtheorem{lemma}[thm]{Lemma}

\newtheorem{corollary}[thm]{Corollary}

\newtheorem{proposition}[thm]{Proposition}
\newtheorem{question}[thm]{Question} \newtheorem*{claim*}{Claim}

 \theoremstyle{definition}
\newtheorem{defn}[thm]{Definition}
\newtheorem{defns}[thm]{Definitions} 
 \newtheorem{ex}[thm]{Example}

\newtheorem{remarks}[thm]{Remarks} 
\newtheorem{remark}[thm]{Remark}
\theoremstyle{remark}

\begin{document}

\maketitle


\def\HDS{half-disk sum}

\def\Length{\text{Length}}

\def\Area{\text{Area}}
\def\Im{\text{Im}}
\def\im{\text{Im}}
\def\cl{\text{cl}}
\def\rel{\text{ rel }}
\def\irred{irreducible}
\def\half{spinal pair }
\def\spinal{\half}
\def\spinals{\halfs}
\def\halfs{spinal pairs }
\def\reals{\mathbb R}
\def\rationals{\mathbb Q}
\def\complex{\mathbb C}
\def\naturals{\mathbb N}
\def\integers{\mathbb Z}
\def\id{\text{id}}
\def\Chi{\raise1.5pt \hbox{$\chi$}}

\def\proj{P}
\def\hyp {\hbox {\rm {H \kern -2.8ex I}\kern 1.15ex}}

\def\Diff{\text{Diff}}

\def\weight#1#2#3{{#1}\raise2.5pt\hbox{$\centerdot$}\left({#2},{#3}\right)}
\def\intr{{\rm int}}
\def\inter{\ \raise4pt\hbox{$^\circ$}\kern -1.6ex}
\def\Cal{\cal}
\def\from{:}
\def\inverse{^{-1}}
\def\Max{{\rm Max}}
\def\Min{{\rm Min}}
\def\fr{{\rm fr}}
\def\embed{\hookrightarrow}
\def\Genus{{\rm Genus}}
\def\Z{Z}
\def\X{X}

\def\roster{\begin{enumerate}}
\def\endroster{\end{enumerate}}
\def\intersect{\cap}
\def\definition{\begin{defn}}
\def\enddefinition{\end{defn}}
\def\subhead{\subsection\{}
\def\theorem{thm}
\def\endsubhead{\}}
\def\head{\section\{}
\def\endhead{\}}
\def\example{\begin{ex}}
\def\endexample{\end{ex}}
\def\ves{\vs}
\def\mZ{{\mathbb Z}}
\def\M{M(\Phi)}
\def\bdry{\partial}
\def\hop{\vskip 0.15in}
\def\hip{\vskip0.1in}
\def\mathring{\inter}
\def\trip{\vskip 0.09in}
\def\PML{\mathscr{PML}}
\def\J{\mathscr{J}}
\def\G{\mathscr{G}}
\def\H{\mathscr{H}}
\def\C{\mathscr{C}}
\def\S{\mathscr{S}}
\def\T{\mathscr{T}}
\def\E{\mathscr{E}}
\def\K{\mathscr{K}}
\def\L{\mathscr{L}}
\def\suchthat{|}
\newcommand\invlimit{\varprojlim}
\newcommand\congruent{\equiv}
\newcommand\modulo[1]{\pmod{#1}}
\def\ML{\mathscr{ML}}
\def\Stack{\mathscr{T}}
\def\M{\mathscr{M}}
\def\A{\mathscr{A}}
\def\union{\cup}
\def\atlas{\mathscr{A}}
\def\interior{\text{Int}}
\def\frontier{\text{Fr}}
\def\composed{\circ}
\def\GL{\mathscr{GL}}
\def\PC{\mathscr{PC}}
\def\PM{\mathscr{PM}}
\def\D{\mathscr{D}}
\def\Proj{\mathscr{P}}
\def\PW{\mathscr{PW}}
\def\F{\mathscr{F}}
\def\FH{\mathscr{FH}}
\def\FM{\mathscr{FM}}
\def\PFM{\mathscr{PFM}}
\def\PF{\mathscr{PF}}
\def\FD{\mathscr{FD}}
\def\FDM{\mathscr{FDM}}
\def\FDW{\mathscr{FDW}}
\def\PFDW{\mathscr{PFDW}}
\def\PFD{\mathscr{PFD}}
\def\PIFD{\mathscr{PIFD}}
\def\PQW{\mathscr{PQW}}
\def\PW{\mathscr{PW}}
\def\WC{\mathscr{WC}}
\def\W{\mathscr{W}}
\def\WM{\mathscr{WM}}
\def\PWM{\mathscr{PWM}}
\def\PFDM{\mathscr{PFDM}}
\def\D{\mathscr{D}}
\def\B{\mathscr{B}}
\def\I{\mathscr{I}}
\def\level{\mathfrak{L}}
\def\real{\mathfrak{R}}
\def\FDL{\mathscr{FDL}}
\def\FDM{\mathscr{FDM}}
\def\PFDL{\mathscr{PFDL}}
\def\FDLM{\mathscr{FDLM}}
\def\PFDLM{\mathscr{PFDLM}}
\def\CFD{\mathscr{QFD}}
\def\FC{\mathscr{FC}}
\def\PFC{\mathscr{PFC}}
\def\FCM{\mathscr{FCM}}
\def\PFCM{\mathscr{PFCM}}
\def\PFC{\mathscr{PFC}}
\def\FT{\mathscr{FT}}
\def\FCT{\mathscr{FCT}}
\def\PFDT{\mathscr{PFDT}}
\def\PFCDT{\mathscr{PCFDT}}
\def\PFCD{\mathscr{PCFD}}
\def\PVC{\mathscr{PVC}}
\def\VC{\mathscr{VC}}

\def\WC{\mathscr{WC}}
\def\PI{\mathscr{PI}}
\def\PQI{\mathscr{PQI}}
\def\VM{\mathscr{VM}}
\def\PVM{\mathscr{PVM}}
\def\DG{\breve{\mathscr{G}}}
\def\SG{\mathfrak{S}}
\def\V{\mathscr{V}}
\def\LV{\mathscr{LV}}
\def\PV{\mathscr{PV}}
\def\QV{\mathscr{QV}}
\def\PQV{\mathscr{PQV}}
\def\abs{\odot}
\def\DS{\breve S}
\def\DL{\breve L}
\def\DBL{\breve{\bar L}}
\def\II{[0,\infty]}
\def\equiv{\hskip -3pt \sim}
\def\nat{{\mathbb{N}_0}}
\def\Infty{\hbox{$\infty$\kern -8.1pt\raisebox{0.2pt}{--}\kern 1pt}}

\def\split{\prec}
\def\pinch{\succ}
\def\FB{\mathbb F}
\def\SB{\mathbb S}
\def\BP{\mathbb P}
\def\SBB{{\mathbb S}\kern -6pt\raisebox{1.3pt}{--} \kern 2pt}

\centerline{\bf PRELIMINARY}
\vskip 0.3in
\begin{abstract} We describe spaces of
essential finite height (measured) laminations in a surface $S$ using a parameter space we call $\SB$, an ordered semi-ring.  We show that for every finite height essential lamination $L$ in $S$, there is an action of $\pi_1(S)$ on an $\SB$-tree dual to the lift of $L$ to the universal cover of $S$.  
 \end{abstract}

\section{Introduction}

In a previous paper \cite{UO:Pair} we defined ``finite depth essential measured laminations" in surfaces, and prepared the way for this paper.  In this paper, we will eventually change terminology slightly, we will
to describe a suitably projectivized space of ``finite height measured laminations," contained in a larger projectivized space of  ``finite height essential laminations."

\begin{defn}\label{FiniteDepthLam}  A {\it finite depth essential measured lamination} in a closed surface $F$ is an isotopy class of laminations of the form $\displaystyle L=\bigcup_{j=0}^dL_j$ where $(L_0,L_1,\ldots, L_d)$  is finite sequence of measured laminations with each $L_i$ embedded in $\displaystyle\hat F_i= F\setminus \bigcup_{j<i}L_j$ with transverse measure $\mu_i$ and without leaves isotopic to leaves of $\bigcup_{j<i}L_j$,  and such that $\displaystyle \bigcup_{j\le i}L_j$ is an essential lamination in $F$ for each $i\le d$.     \end{defn}

To understand these laminations, we can begin by understanding the {\it depth $i$ lamination} $L_i$.  This lamination $L_i$ is embedded in a possibly non-compact surface $F_i$ with infinite boundary cusps, namely the completion of the surface 
$\displaystyle \hat F_i =F\setminus \bigcup_{j<i}L_j$.  We show a typical component of $F_i$ (relabeled $S$) in Figure \ref{DepthSurfPair}(a), a surface with infinite outward cusps on its boundary.  It is sometimes convenient to truncate the outward boundary cusps.   Some ends of leaves of the measured lamination $L_i$ extend into the cusps, so when we truncate $F_i=S$ we obtain an essential measured lamination with boundary in the arcs of truncation, which are bold in the figure.  In general, we will deal with surfaces $ S$ with infinite interior cusps and infinite boundary cusps.  Truncating the cusps and boundary cusps we obtain a {\it surface pair} $(\bar S, \alpha)$ as show in Figure \ref{DepthSurfPair}(b).  {\it Interior cusps} correspond to closed curves of $\alpha$, {\it boundary cusps} correspond to arcs of $\alpha$.  The surface $S$ may have some closed boundary components (coming from closed leaves of some $L_j$) which remain in $(\bar S, \alpha)$ as closed boundary components disjoint from $\alpha$.  We will often use $\delta$ to denote the closure in $\bdry \bar S$ of $\bdry \bar S\setminus \alpha$.

From the point of view of hyperbolic geometry, the surface-with-cusps $S$ is the ``wrong" surface.  There is a {\it dual surface-with-cusps} $\DS$ which, when given a hyperbolic structure, gives a closer approximation to the finite depth measured laminations.  The cusps of $S$ become boundary curves in $\DS$ and the boundary curves of $S$ become cusps of $\DS$, see Figure  \ref{DepthSurfPair}(c).   We may suppose there is a fixed given homeomorphism between the interiors of $S$ and $\breve S$ which does not extend to boundaries.  The homeomorphism gives a correspondence between well-behaved curves in the two surfaces.  A well-behaved bi-infinite embedded curve in $S$ corresponds to a properly embedded arc in $\breve S$, and so on.

\begin{figure}[ht]
\centering
\scalebox{1}{\includegraphics{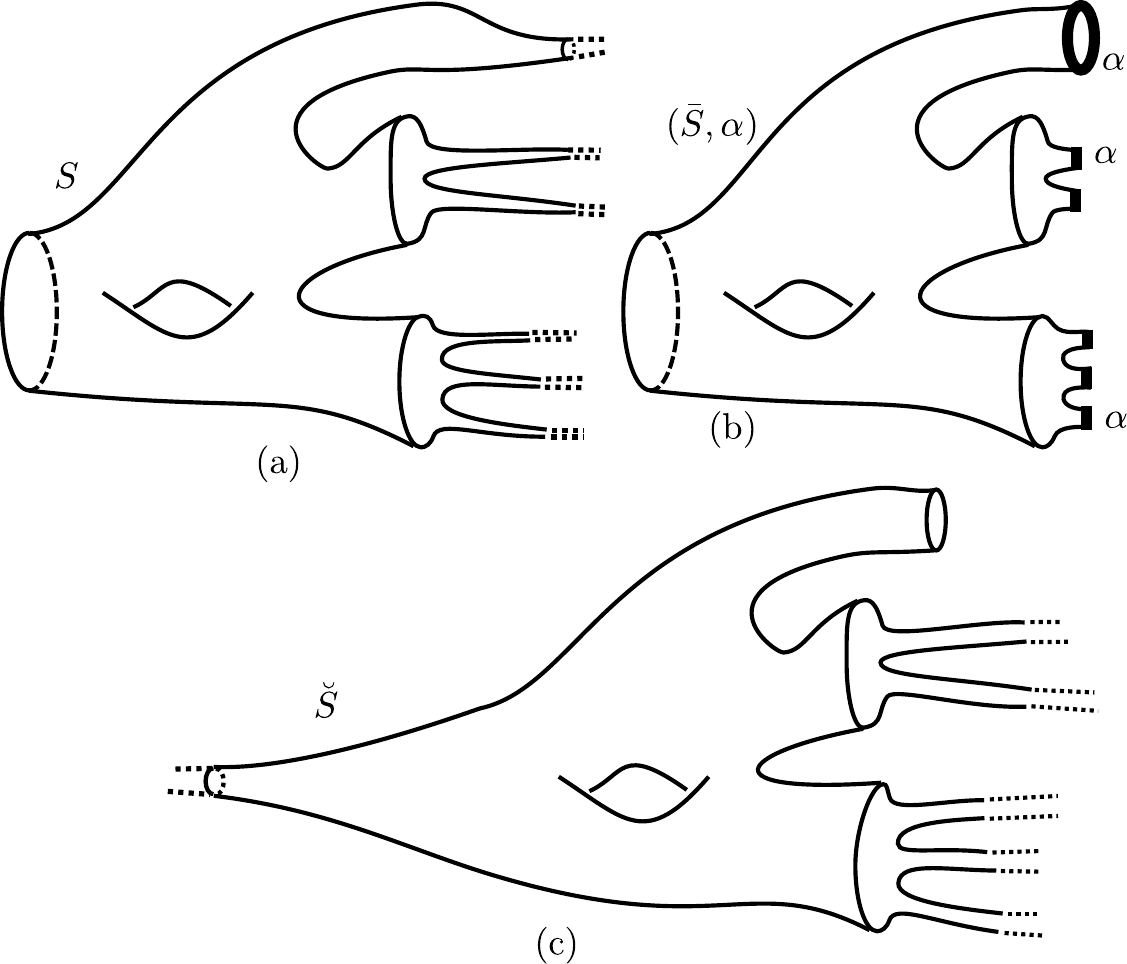}}
\caption{\small Surface $ S$ with cusps, the corresponding surface pair $(\bar S, \alpha)$, and the dual surface $\breve S$.}
\label{DepthSurfPair}
\end{figure}

Throughout the paper, we will assume that the geometric Euler characteristic of surface pairs satisfies $\Chi_g(\bar S, \alpha)<0$, which implies that $ S$ and the dual $\DS$ have complete hyperbolic structures.  We will also assume that $S$ is oriented.  Since $S$ and $\DS$ are associated to $(\bar S,\alpha)$ uniquely, we may define $\Chi_g(S)=\Chi_g(\DS)=\Chi_g=\Chi_g(\bar S,\alpha)$.   See \cite{UO:Pair} for a definition of  $\Chi_g$.

Our goal in this paper is to define and describe a space of finite depth measured laminations.  For many different reasons, to define such a space, one is forced to enlarge the set of laminations considered.  The larger set of laminations is the set of ``finite depth" laminations.  A finite depth lamination has a ``finite depth structure," but the lamination at each level need not be measured in the usual sense.  
We will allow transverse measures which are not required to be locally finite.  This means that a lamination at each level is only ``partially measured."  

The following is a precise definition.  For technical reasons, in our setting it makes more sense to speak of ``finite height laminations."   The terminology ``finite depth" is commonly used in foliation theory.

\begin{defn}\label{FDepthDef}  A {\it finite height (depth) essential lamination} in a surface $S$ with cusps is the isotopy class of an essential laminations $\displaystyle L=\bigcup_{j=0}^hL_j$ where $L_j=X_j\cup Y_j$ satisfying:

\begin{tightenum}
\item  $L_i$ a lamination embedded in $\displaystyle\hat S_i=S\setminus \bigcup_{j>i}L_i$ with no $\bdry$-parallel leaves;

\item  $Y_i$ a lamination (with everywhere locally infinite transverse measure) embedded in $\displaystyle\hat S_i= S\setminus \bigcup_{j>i}L_i$;

\item $X_i$ is a measured lamination, with transverse real measure $\mu_i$ of full support, embedded in $\displaystyle\grave S_i= S\setminus\left[ Y_i\cup\left(\bigcup_{j>i}L_i\right)\right]$;

\item $\bar L_i=L_i\cup \bdry S_i$ is a  lamination in the completion $S_i$ of $\hat S_i$, which is a surface with cusps;

\item $X_i$ is a measured lamination (in the sense of \cite{UO:Pair}) with transverse real measure $\mu_i$ embedded in the completion of $\displaystyle\grave S_i$.
\end{tightenum}
We say $L$ is a {\it total height $h>0$ lamination} if $L_0$ and $L_h$ are non-empty.

We say  $\bar L_i=L_i\cup \bdry S_i$ is the {\it completion of $L_i$}.  Similarly, we define {\it completions} $\bar Y_i=Y_i\cup \bdry S_i$ and $\bar X_i$ being the union of $X_i$ and the boundary of the completion of $\grave S_i$.

For simplicity, throughout most of this paper we assume all underlying laminations can be represented as geodesic laminations in $S$ with some hyperbolic structure.  This means that in some cases transverse measures must be atomic.  On the other hand, in the last section about actions on trees, we must avoid leaves with atomic measures, so there we will replace leaves with atomic measures by families of isotopic leaves with non-atomic transverse measures.
 \end{defn}
 
 One reason to study finite height measured laminations, is that any essential lamination admits a structure as a finite height measured lamination.   This means it admits a structure as a finite height lamination with $Y_j=\emptyset$ for each $j$ in the above definition.
 
 \begin{proposition} \label{PlanteProp}  Suppose $L$ is an essential lamination in $S$ (geodesic lamination in $S$).  Then $L$ admits a structure as a finite height measured lamination, $\displaystyle L=\bigcup_{j=0}^hL_j$.   This means in particular that each $L_j$ has a transverse $\reals$-measure of full support.
 \end{proposition}

We note that $L_i$ is a lamination in $\hat S_i$ but it may not be closed in the completion $S_i$, because leaves of $L_i$ may spiral to a closed curve of $\bdry S_i$, but $\bar L_i$ is a lamination in $S_i$.  We will see that if $L_i$ is a measured lamination in $\hat S_i$, then $\bar L_i$ has a transverse measure which is locally finite on transversals in $\hat S_i$, but is infinite on transversals having one endpoint in $\bdry S_i$.   We say this kind of transverse measure is a {\it geometric measure} on $\bar L_i$, to be defined later.

In Section \ref{Parameter} we will define a parameter space $\SB$, and we will define Borel measures with values in this parameter space.   As we shall see, essential laminations with suitable transverse measures in this parameter space are finite height essential laminations. The parameter space is sufficiently closely related to the non-negative extended reals to allow us to imitate many of the features of the construction of the measured lamination space of a surface pair, which is defined in terms of ordinary finite measures on transversals, as in \cite{UO:Pair}.  An $\SB$-measured lamination will be defined like a measured lamination, or $\reals$-measured lamination;  it has an invariant measure on transversals with values in $\SB$.  We will show that every $\SB$-measured lamination is a finite height lamination, see Lemma \ref{DefnsLemma}.

It turns out that $\SB$-measures on an essential lamination can be related to each other by a certain ``level alignment" equivalence relation.  Definition \ref{FDepthDef} does not distinguish between different representatives.  A ``proximal representative lamination" is obtained roughly by moving laminations $L_i$ to a lower level if possible. In fact, we will also use a stronger equivalence relation, ``level adjustment" from which one obtains a ``contiguous representative" by also moving {\em components} of laminations $L_i$ to a higher level when possible.  We will let $\F(S)$ denote the set of pairs $(L,\nu)$, where $L$ is a proximal finite height essential lamination and $\nu$ is the associated $\SB$-measure.  Similarly, we let $\FC(S)$ denote the set of  contiguous representatives, which of course is a subset of $\F(S)$.  Then we will topologize the sets $\F(S)$ and $\FC(S)$ very roughly as follows. Let $S$ be the surface-with-cusps with $\chi_g(S)<0$ and with a chosen complete hyperbolic structure.  Let $\G$ denote the set of geodesics in $S$, including geodesics of $\bdry S$.  The set $\G$ does not include any arc geodesics with boundary in $\bdry S$.  For technical reasons, $\G$ will include multiple copies of some geodesics, some copies being oriented.   We use intersection numbers $i_\theta((L,\nu))\in \SB$ with elements of $\theta\in \G$ to topologize $\F(S)$
 
\begin{thm} \label{SpacesThm} The set $\F(S)$ of proximal representatives of $\SB$-measured laminations in a surface-with-cusps $S$ is in bijective correspondence with the points of a subspace $\F(S)$ of $\SB^\G$.  If we also projectivize in $\SB^\G$ we obtain the space  $\PF(S)$.  Also, there are subspaces $\FM(S)\subset \F(S)$ and $\PFM(S)\subset \PF(S)$ of (projectivized)  finite height laminations which are $\reals$-measured at each level. 
\hip

\noindent In the same way, we obtain a subspace $\FC(S)\subset\F(S)\subset \SB^\G$, a projectivized version $\PFC(S)$ and subspaces $\FCM(S)$ of $\FC(S)$ and $\PFCM(S)$ of $\PFC(S)$.
\end{thm}

The names of these spaces should be read as follows:   $\F(S)$ is the {\it space of finite height laminations in $S$};  the space $\PF(S)$ is the {\it projective space of finite height laminations in $S$};  the space $\FM(S)$  is the {\it space of finite height measured laminations in $S$}; the space $\PFM(S)$ is the {\it projective space of finite height measured laminations in $S$}.  Similarly,  $\FC(S)$ is the {\it space of finite height contiguous laminations in $S$};  the space $\PFC(S)$ is the {\it projective space of finite height contiguous laminations in $S$}; etc.

The dual point of view is equally good, and in some ways preferable.  If $\DS$ is the dual surface endowed with a hyperbolic structure, then for every essential lamination $L$ in $S$ we have a corresponding dual lamination $\DL$ in $ \DS$.   Leaves of $L$ joining cusps become geodesic leaves of $\DL$ joining corresponding boundary components of $\DS$.   Closed leaves remain closed leaves. Since $\bar L$ contains boundary-parallel leaves, some leaves of $\DBL$ isotope into cusps, and can be represented by horocycles or horocycle segments.   Thus the lamination $\DBL$ can be realized as a  lamination each of whose leaves is a  geodesic or a (segment of) a horocycle, where the (segments of) horocycles arise from boundary-parallel leaves of $\bar L$.   All other leaves of $\bar L$ become geodesics, possibly with boundary in $\bdry \DS$.  No leaf of $\DBL$ has an end in a boundary cusp of $\DS$, but ends of leaves in $\DL$ can be mapped to interior cusps of $\DS$, corresponding to leaves of $L$ spiraling to a component of $\bdry S$.    

The topology of a space of laminations with locally everywhere infinite transverse measures can be considered in a context independent of finite height structures.   Let $\GL(S)$ denote the set of geodesic laminations in $S$.  Let $\G$ denote the set of embedded geodesics in $S$, with some geodesics represented more than once, possibly with orientations, as before.  Let $\{0,\infty\}$ be a discrete space with two elements.  For every $L\in \GL(S)$ and every $\gamma\in \G$, $\gamma$ a geodesic, let $i_\gamma(L)=0$ if $\gamma$ does not intersect $L$ transversely; let $i_\gamma(L)=\infty$ if $\gamma$ intersects $L$ transversely.  In the case of closed geodesics $\gamma$, each of which appears multiple times in $\G$, the intersection numbers will be appropriately defined later.
 
\begin{thm} \label{GeodesicLamThm} The map $L \mapsto (i_\theta (L))_{\theta\in \G}\subset \{0,\infty\}^\G$ is an injective map $\GL(S)\to \{0,\infty\}^\G$, so $\GL(S)$ obtains the subspace topology in $\{0,\infty\}^\G$; $\GL(S)$ is totally disconnected.  The space $\GL(S)$ is homeomorphic to a subspace of $\F(S)$ contained in the subspace of total height 0 laminations.
\end{thm}

\begin{question}  What is the relationship of the topology of $\GL(S)$ to the Hausdorff metric topology on geodesic laminations?  Could it be the same?
\end{question}

If one examines the topology of $\PF(S)$ using examples, as we shall do later, one discovers that there is much non-Hausdorff behavior in the spaces.  Suppose we are given any finite height lamination of the form $\displaystyle L=\bigcup_{j=0}^hX_j\cup Y_j$, where $X_j$ is transversely measured and $Y_j$ has everywhere locally infinite transverse measure.  If $L'$ is topologically the same lamination as $L$, but with the transverse locally finite measures on $X_j$ (partly) extended to $Y_j$, then it turns out that the projective class of $L'$ cannot be separated in $\PF(S)$ from the class of $L$ by a pair of disjoint open neighborhoods of $L'$ and $L$.  Here if  $\displaystyle L'=\bigcup_{j=0}^hX_j'\cup Y_j'$, we are assuming that for some $j$, $X_j'\supsetneqq X_j$, with the measure on $X_j'$ an extension of the measure on $X_j$.  

Although the topology of $\PF(S)$ is somewhat pathological, the space can, to some extent, be understood in terms of more familiar spaces.  If $S$ is a surface with cusps, the measured lamination space $\M(S)$ and the projective measured lamination space $\PM(S)$ were calculated in \cite{UO:Pair}.  Note, however, that in \cite{UO:Pair}, the definition of the measured lamination space is not quite standard.   In this paper, we will define enlargements $\WM(S)$ and $\PWM(S)$ of these spaces whose elements are (projective classes of) measured laminations with transverse measures which are not everywhere locally finite, i.e., transverse measures with values in $[0,\infty]\subset \bar\reals$.   We can read $\WM(S)$ as {\it the weakly measured lamination space of $S$} and $\PWM(S)$ as {\it the projective weakly measured lamination space of $S$}.  Each of $\WM(S)$ and $\PWM(S)$ contains $\GL(S)$ as a subspace, while $\WM(S)$ contains $\M(S)$ as a subspace and $\PWM(S)$ contains $\PM(S)$.

The following are alternative definitions of $\F(S)$ and $\PF(S)$ .  We will have to show they are equivalent to our previous definitions.

\ \begin{defn} \label{WeakDefn} Let $S$ be a surface-with-cusps satisfying $\Chi_g(S)<0$.  The space $\F(S)$ is a subspace $\F(S)\embed \WM(S)^\nat$: $$\F(S)=\{(W_0,W_1,W_2,\ldots,W_h,\emptyset, \emptyset, \ldots)\in \WM(S)^\nat: \text{ satisfying conditions (i)-(iii) below}\}:$$

\begin{tightenum}
\item $W_h\ne \emptyset$, $W_i=\emptyset$ for $i>h$,
\item $W_{i}\subsetneqq W_{i-1}$ for $i\le h$ ,
\item $W_i$ has weak measure $\rho_i$ satisfying $\rho_i(T)=\infty$ if $\rho_{i+1}(T)>0$ for any transversal $T$.
\end{tightenum}

If we simultaneously projectivize all measures $\rho_i$, so that   $$[(W_0,\rho_0),(W_1,\rho_1),\ldots]\sim [(W_{0},\lambda\rho_{0}),(W_1,\lambda\rho_{1}),\ldots]$$ for any $\lambda>0$, we obtain a quotient called the projective space $\PF(S)$. 

If we replace condition (iii) by 

\hip

(iii$'$) $W_i$ has weak measure $\rho_i$ satisfying $\rho_i(T)=\infty$ if and only if $\rho_{i+1}(T)>0$ for any open transversal $T$,
\hip
\noindent then we obtain spaces of laminations which are $\reals$-measured at each level, namely $\FM(S)$ and $\PFM(S)$.  

\end{defn}

The relationship between the laminations $W_i$ for a proximal representative and the laminations $L_i$ in Definition \ref{FDepthDef} is the following:  $L_h=W_h$ is a lamination in $S$, $L_i=W_i\setminus W_{i+1}$ is a weakly measured lamination in the completion $S_i$ of $\hat S_i=S\setminus W_{i+1}$.  

\begin{thm} \label{DefnsThm} The spaces described in Theorem \ref{SpacesThm} are homeomorphic to the spaces redefined in Definition \ref{WeakDefn}.
\end{thm}

\begin{corollary} \label{MeasuredRedefineCor} The space $\FM(S)$ of finite height measured laminations is the subspace of  $\F(S)\embed \WM(S)^\nat$
 consisting of elements $(W_0,W_1,W_2,\ldots)$ such that $L_i=W_i\setminus W_{i+1}$ are transversely measured with a (locally finite) transverse $\reals$-measure in $\hat S_{i} =S\setminus W_{i+1}$.
\end{corollary}

We will make a connection with the space $\M(S)$ of measured laminations in $S$, as described in \cite{UO:Pair}.   Recall that $\M(F)=\M(\bar F,\alpha)$ and the projectivized $\PM(F)=\PM(\bar F,\alpha)$ were calculated in \cite{UO:Pair} with the following result, which is more easily stated for the truncated version of $F$:

\begin{thm}\label{MeasuredSpaceThm} (From \cite{UO:Pair}.) Suppose $(\bar F, \alpha)$ is a connected surface pair satisfying $\Chi_g=\Chi_g(\bar F, \alpha)<0$, with topological Euler characteristic $\Chi(\bar F)=\Chi$.  Suppose $\alpha$ contains $b$ closed curves and $c$ arcs.
Then $\M(\bar F, \alpha)$ is homeomorphic (via a homeomorphism linear on projective equivalence classes) to $\reals^{-3\Chi-b+c}\times [0,\infty]^b=\reals^{-3\Chi_g-c/2-b}\times [0,\infty]^b$, where $[0,\infty]$ denotes $[0,\infty]\subset \bar\reals$, a subset of the extended reals.  Thus $\PM(\bar F, \alpha)$ is homeomorphic to the join of a sphere $S^{-3\Chi-b+c-1}=S^{-3\Chi_g-c/2-b-1}$ and a simplex $\Delta^{b-1}$.
\end{thm}

\begin{defn}   If $L=(L_0,L_1,\ldots L_h)$ is a finite height lamination, an element of $\F(S)$, let $p(L)=(L_1,\ldots L_h)$. The induced map on $\PF(S)$ is also called $p$.


 \end{defn}
 
 Simply put,  $p$ erases the lowest level lamination and moves higher level laminations down by one level.  Clearly $p^k$ erases the lowest $k$ levels, $i=0,1,\ldots, k-1$.  If we represent a lamination $L$ in terms of weakly measured laminations,
  $L=(W_0,W_1,\ldots, W_h, \emptyset,\emptyset,\ldots)$, then $p(L)=(W_1,\ldots, W_h, \emptyset,\emptyset,\ldots)$.

\begin{proposition}  \label{MeasProp}  If $K$ is a fixed finite height lamination $K=(K_0,K_1,\ldots K_h)\in \F(S)\setminus\{\emptyset\}$, then $p\inverse(K)\subset \F(S)$ is homeomorphic to $\WM(F)$ where $F$ is the completion of $S\setminus K$, and contains a subspace homeomorphic to $\M(F)$.  For $p:\FM(S)\to \FM(S)$, if $K\in \FM(S)\setminus\{\emptyset\}$, $p\inverse(K)$ is homeomorphic to $\M(F)$.  Similar statements hold for projective spaces $\PF(S)$ and $\PFM(S)$.
\end{proposition}

The proposition shows that our spaces of finite height laminations contain subspaces of measured laminations whose topological types we know. 

\begin{corollary}  $\F(S)$ contains a subspace homeomorphic to $\M(S)$, namely the subspace of total height 0 laminations in $\F(S)$.  Similar statements can be made for $\PF(S)$, $\FM(S)$, and $\PFM(S)$.
\hip\noindent For every $K\in \F(S)$, $ \F(S)$ contains a subspace homeomorphic to $\M(F)$, where $F$ is the completion of $S\setminus K$.  Similar statements can be made for $\PF(S)$, $\FM(S)$, and $\PFM(S)$.
\end{corollary}

We will prove a result about finite height invariant measured laminations for automorphisms of surfaces.  The following theorem is not a significant improvement on the Nielsen-Thurston theory of automorphisms of surfaces, but its proof might serve as a model for constructing invariant finite height measured laminations in other settings.  

\begin{thm} \label{InvariantThm} Suppose $S$ is a surface with cusps satisfying $\chi_g(S)<0$.  Suppose $f:S\to S$ is an orientation preserving automorphism (self-homeomorphism up to isotopy), and suppose that if $S$ is closed then $f$ is not periodic.   Then there exists an invariant finite height measured lamination $L=\bigcup_{j=0}^hL_j$, where $L_j$ has transverse measure $\mu_j$, and there exist eigenvalues $\lambda_j> 0$ such that $f((L_j,\mu_j))=(L_j,\lambda_j\mu_j)$.  Further, the lamination $L$ fills the surface, in the sense that completions of components of $S\setminus L$ are disks with boundary cusps.  
\end{thm}

In Section \ref{TreesSection} we investigate $\SB$-trees.   There are two notions of $\SB$-trees:  an $\SB$-tree is an order tree with compatible $\SB$-measures on segments; a ``metric $\SB$-tree" also admits an $\SB$-metric.  We investigate the connection between finite height laminations and actions on $\SB$-trees.  For metric $\SB$-trees, we must modify the laminations slightly.  Let $(L,\nu)$ be a finite height lamination in a surface $S$ (with boundary and cusps), then we modify $(L,\nu)$ in a standard way to obtain another $\SB$-lamination $(\bar L,\bar \nu)$, which we call a {\it geometric $\SB$-lamination}.

\begin{thm} \label{TreeThm} Let $(L,\nu)$ be a finite height lamination in a surface $S$ (with boundary and cusps) and let $(\bar L,\bar \nu)$ be the corresponding geometric lamination.  The lift $(\tilde L,\tilde \nu)$ is dual to an $\SB$-tree $\T$ and there is an action of $\pi_1(S)$ on $\T$ which preserves the measure.  

The lift $(\tilde {\bar L},\tilde{ \bar\nu})$ of $(\bar L,\bar \nu)$ to the universal cover $\tilde S$ of $S$ is dual to a metric $\SB$-tree $\bar \T$.  $\bar \T$ is an $\SB$-metric space, and there is an action of $\pi_1(S)$ on $\bar\T$ which preserves the metric.  
\end{thm}

One piece of unfinished business here is to obtain a finite height lamination from an action of $\pi_1(S)$ on an $\SB$-tree.  

The following are motivations and goals related to this paper:

\hop
(1) A natural goal is to extend the space $\PF(S)$ to make a connection with hyperbolic geometry.  We have designed our finite height lamination spaces to make the enlargement possible.  Let us call the proposed enlarged space {\it the finite height Teichm\"uller space of $S$, $\FT(S)$}.

We give a {\em highly speculative} description of the space here.  A point $H$ in this space will be represented by the projective class of a finite height measured lamination
$\displaystyle L=\bigcup_{j=1}^hL_j$, together with a (not necessarily complete) hyperbolic structure $h$ on the dual $\breve F$ of the completion $F$ of $S\setminus L$. As usual, two hyperbolic structures on $F$ are considered equivalent if one can be pulled back from the other by a homeomorphism isotopic to the identity.  Note that we have put the hyperbolic structure at level 0 now, and the finite height measured lamination occupies levels 1 to $h$.  

When a hyperbolic structure on $\breve F$ degenerates, we obtain a measured lamination in $F$.  So we see that a point in $\FT(S)$ can be thought of as the result of repeated degenerations:  Starting with a hyperbolic structure on $\breve S$, a degeneration of the hyperbolic structure yields a measured lamination $L_1$ in $S$.  Next suppose that we consider a degeneration of the hyperbolic structure on the dual $\breve F$ of the completion $F$ of $S\setminus L_1$.  We relabel $L_1$ calling it $L_2$ instead, pushing it to a higher level, and we obtain a new measured lamination $L_1$ in $F$.  Then we redefine $F$ as the completion of $S\setminus(L_1\cup L_2)$ and consider a degeneration of $\breve F$ again, and so on.   

Either $\FT(S)$ should contain $\PF(S)$, or possibly $\PFC(S)$.

The required technicalities for establishing the existence of these spaces may turn out to be fairly routine.  In order to describe topologies, one again defines intersection/length functions $i_\theta$, $\theta\in \G$ which assigns a ``length" in $\SB$ to every $\theta\in \G$ viewed as a curve in $H$, which is a multiply degenerate hyperbolic surface.  For example, if a non-oriented $\gamma\in \G$ lies entirely in $F$, we define $i_\gamma(H)$ to be the element of $\SB$ such that the level $\level(i_\gamma(H))=0$ and the real part $\real(i_\gamma(H))$ is the length of the curve in $\breve F$, see Definition \ref{BLDefn}.  If a non-oriented curve intersects $L$ transversely, then the intersection number has level equal to the maximum level $m$ of the $L_i$'s intersected, with real part equal to the intersection of $\gamma$ with $L_m$.

 The degeneration of $F$ can be understood (except when a component of $F$ is closed) using a method due to William Thurston, depending on a choice of ideal triangulation of $\breve F$.  Each ideal triangle is equipped with the standard partial horocyclic foliation.   Combining these, one obtains a partial foliation for all of $\breve F$.  Removing horocycles and horocycle segments, one obtains a measured lamination $\breve L_i$ in $\breve F$.  The lamination may have boundary in $\bdry \breve F$, and it has noncompact leaves only if these end in interior cusps of $\breve F$, which can only happen if the hyperbolic structure on $\breve F$ is non-complete.   Dualizing, we obtain a measured lamination in $F$;  leaves ending in a cusp of $\breve F$ become leaves that spiral towards a closed leaf of $\bdry F$, with the same sense of spiraling as in $F$.   A boundary point of a leaf in $\bdry \breve F$, becomes an end of a leaf in a cusp or boundary cusp in $F$.   We observe that allowing measured laminations with leaves spiraling to curves of $\bdry F$ corresponds to allowing non-complete hyperbolic structures for $\breve F$.
 
Finally we observe that an element $H$ of $\FT(S)$ corresponds to an action of $\pi_1(S)$ on an object which is a hybrid of a tree and hyperbolic space.  The element $H$ is a finite height measured lamination $L$ in $S$ together with a hyperbolic structure on the dual $\breve F$ of the completion of $S\setminus L$.  Lifting the lamination to the universal cover $\tilde S$, we have a finite height measured lamination $\tilde L$ in $\tilde S$, together with a hyperbolic structure on the completion of the dual of $\tilde S\setminus \tilde L$.  Now we replace the completion of each component of $\tilde S\setminus \tilde L$ by its dual surface compactified by adding one point at each cusp, and insert these dual surfaces into the $\SB$-tree dual to $\tilde L$ in $\tilde S$ at the points of the $\SB$-tree dual to $\tilde L$ at the appropriate points of the dual tree.   We will describe an analogous insertion of subtrees in Section \ref{TreesSection}.  The resulting tree-like object is an $\SB$-metric space, and the action of $\pi_1(S)$ of course respects this metric.

\hop
(2) The theory of measured lamination spaces in 3-manifolds, see \cite{AH:MeasuredLaminations} and \cite{UO:MeasuredLaminations}, is incomplete and not satisfying.  It may be easier to understand spaces of finite height (measured) laminations in 3-manifolds.
\hop
(3) In a weak sense the 2-dimensional invariant laminations for generic automorphisms of handlebodies, see \cite{UO:Autos}, \cite{LC:Tightness}, yield finite height measured invariant laminations.  Here the level 0 lamination is 1-dimensional and contained in the boundary of the handlebody , an invariant measured lamination for the pseudo-Anosov restriction to the boundary, while the level -1 lamination is 2-dimensional and is embedded in the interior of the handlebody.  For reducible automorphisms of handlebodies there may be examples of invariant finite height measured laminations of arbitrarily large height.  
\hop
(4) Finite height measured laminations appear implicitly in the theory of $\text{Out}(F_n)$, the group of automorphisms of a free group of rank $n$.  It might be possible to reformulate this theory by using a version of Culler-Vogtmann Outer Space corresponding to graphs with $\SB$-metrics  rather than ordinary $\reals$-metrics.

\hop
(5) For our purposes in this paper, $\SB$ is the best abelian semi-group to use as a transverse structure or parameter space.   In fact, there is a wide choice of abelian semi-groups which could be used in this way, and one can investigate the associated laminations, metric spaces, trees, and actions on trees, see \cite{UO:Order}.  Actions on $\Lambda$-trees, where $\Lambda$ is an ordered abelian group, have been studied extensively.  By using ordered abelian semi-groups instead, we lose algebra but gain flexibility.

\section{A parameter space $\SB$ and finite height $\SB$-measures}\label{Parameter}

In order to define topologies on $\F(S)$, $\FC(S)$, $\PF(S)$, and $\PFC(S)$ we will use a certain parameter space, which we call $\SB$.

\begin{defn}\label{BLDefn}  Let $\SB$ denote the set $\{0\}\cup(\nat\times (0,\infty])$ where $(0,\infty]\subset\bar \reals$ and where $\bar\reals=[-\infty,\infty]$ denotes the extended real line.  We use the
lexicographical order relation on $\nat\times (0,\infty]$, so that $(i,t)<(j,s)$ if either $i<j$ and $s,t\in (0,\infty]$ or $i=j$ and $t<s$.   The element $0\in\SB$ is a least element, $0<(i,t)$ for all $(i,t)\in(\nat\times (0,\infty])$
We make $\SB$ a topological space with the order topology. 
 
 The elements $(i,\infty)$ are called {\it infinities} of $\SB$.  Define a commutative {\it addition} operation 
on $\SB$ by $$(i,t)+(j,u)
=\begin{cases}(i,t+u) & \mbox{if } i=j \\
                     ( i,t) & \mbox{if } i>j 
                                            \end{cases}
$$
$$    (i,t)+0=(i,t),$$
\noindent with the convention $\infty+a=a+\infty=\infty$ for any $a\in (0,\infty]$. Define a commutative {\it multiplication} on $\SB$ by 
$$(i,t)(j,u)=(i+j, tu),$$ $$0(i,t)=(i,t)0=0,$$ with the convention that $a\infty=\infty a=\infty$ for any $a\in (0,\infty]$.  There is a {\it scalar multiplication}:  If $\lambda\in (0,\infty]\subset \bar\reals$, and $x=(i,t)$, then $\lambda x=(i,\lambda t)$.

We say $(i,t)\in (\nat\times (0,\infty])$ is {\it real} if $i=0$.  When we denote an element of $\SB$ by a single symbol $x=(i,t) \in \SB\setminus \{0\}$, we will use $\level(x)=i$ to denote the {\it level of $x$} and $\real(x)=t$ to denote the {\it real part of $x$}, which lies in $(0,\infty]$.   We make the convention that $\real(0)=0$ and $\level(0)$ is undefined.

We define {\it vectors} in $\SB^n$ as $n$-tuples of elements of $\SB$, and give $\SB^n$ the product topology.   We allow $n$ infinite, giving an infinite product.   If $\lambda\in\SB$ and $w\in \SB^n$, $w=(w_1,\ldots,w_n)$, then
we define {\it scalar multiplication} by $\lambda w=(\lambda w_1, \lambda w_2,\ldots, \lambda w_n)$.   
The {\it lattice points} of $\SB^n$ are vectors of the form $[(i_1,\infty),(i_2,\infty),\ldots ,(i_n,\infty)]$; we use $Z$ to denote the set of lattice points in $\SB^n$.   The {\it origin} is the vector with all entries 0.
A {\it cone} in $\SB^n$ is a subset of $\SB^n$ closed under
scalar multiplication by $\lambda\in \SB$.
\end{defn}

We observe that the subset of $\SB$ which we identify with the positive extended reals, namely $\{(0,t)\in \SB:t \in (0,\infty]\}=\{x\in \SB:\level(x)=0\}$, has the usual addition and multiplication of the positive extended reals, and also 
has the usual topology.
Scalar multiplication by $\lambda\in \SB$ shifts the levels of all entries of a vector by the same integer.  Scalar multiplication by a real does not shift levels of entries.  We have used notation for finite products, but everything applies for infinite products  as well.

It is routine to verify the following:

\begin{lemma} $\SB$ is an ordered semi-ring with multiplicative identity $(0,1)$ and additive identity $0$.
\end{lemma}

Viewing $\SB$ as a topological space with the order topology, we can picture it as shown in Figure \ref{DepthParameter}.

\begin{figure}[ht]
\centering
\scalebox{1}{\includegraphics{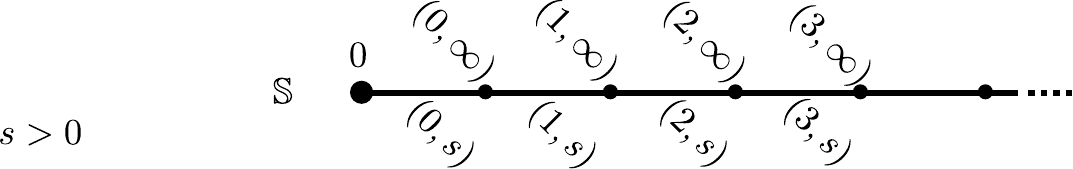}}
\caption{\small The ordered semi-ring  $\SB$ viewed as a topological space.}
\label{DepthParameter}
\end{figure}

\begin{defn}\label{LMeasureDef} A {\it finite height Borel $\SB$-measure} $\nu$ on a Hausdorff topological space $X$ assigns an element $\nu(E)$ of $\SB$ to each Borel set $E$ such that:

(i) there is a uniform bound for $\level(\nu(E))$.

(ii) $\nu(\emptyset)=0$,

(iii) If $\{E_i\}_{i\in I}$ is a countable collection of pairwise disjoint Borel sets in $X$, then 
$$\nu\left(\bigcup_{i\in I} E_i\right )=\sum_{i\in I} \nu(E_i).$$
\end{defn}

Note that the sum in the above definition makes sense.   The abelian semi-group $\SB$ has the least upper bound property, and since we have a uniform bound for $\level(\nu(E))$, the partial sums are bounded.  Even without the finite height condition (i), we could define an $\bar\SB$-measure, where $\bar\SB$ has an adjoined infinity, say $\Infty$, greater than all the other elements, see \cite{UO:Order}.  This yields a Borel $\bar \SB$-measure.

\begin{defn}\label{OpenGradedMeasureDef} A finite height Borel $\SB$-measure  $\nu$ on a Hausdorff topological space $X$ is {\it open-graded} if the union $\bigcup E$ of Borel sets $E$ satisfying $\level (\nu(E))\le k$ or $\nu(E)=0$ is open in $X$. \end{defn} In this paper, an {\it $\SB$-measure} on a Hausdorff space $X$ is assumed to be a finite height Borel measure and open-graded.

The $\SB$-measures we use will satisfy additional conditions.    Associated to an $\SB$-measure $\nu$, we will define a collection of ordinary extended-real valued measures, namely a measure $\nu_k$ associated to each level $k$.  If $E\subset X$ is a Borel set we define

$$\nu_k(E)
=\begin{cases}\real(\nu(E))& \mbox{if } \level(\nu(E))=k \\
                     \infty & \mbox{if }\level(\nu(E))>k\\
                     0& \mbox{if }\level(\nu(E))<k\mbox{ or if }\nu(E)=0\in \SB
                                            \end{cases}
$$

\noindent Recall that the {\it support of the measure $\nu_k$} is the set $Q_k$ of points $x\in X$ having the property that every neighborhood of $x$ has positive measure.   Accordingly, $x\in Q_k$ if and only if every neighborhood $U$ of $x$ satisfies $\level(\nu(U))\ge k$.

\begin{defn}\label{FiniteDepthMeasureDef} A (finite height open-graded) {\it  locally finite $\SB$-measure} on a Hausdorff space $X$ is an $\SB$-measure $\nu$ satisfying the following additional condition:

\begin{enumerate}
\item [(iv)] $\nu_j$ restricted to $Q_j\setminus Q_{j+1}$, which restriction we denote $\hat\nu_j$, is a locally finite measure.   The measure $\hat\nu_j$ is {\it locally finite} if every point has a neighborhood of finite measure.   
\end{enumerate}
\end{defn}

 Notice that the measure $\nu$ can be recovered from the sequence $\{\nu_j\}$ or from the sequence $\{\hat\nu_j\}$.  Namely, to find $\nu(E)$ given all $\nu_i(E)$, let $j$ be the largest $i$ such that $\nu_i(E)>0$, then $\nu(E)=(j,\nu_j(E))\in \SB$.  To recover $\nu$ from $\hat \nu_j$, we need the open-graded condition. We claim that if $j$ is the largest $i$ such that $\hat\nu_i(E\setminus Q_{i+1})>0$, then $\nu(E)=(j,\hat\nu_j(E\setminus Q_{j+1})$.   This is true because if $\level(\nu(E))=j$, then $\nu_j(E)=\real(\nu(E))$, and  if $\nu$ is open-graded, $E$ is contained in the complement of $Q_{j+1}$, so $\hat\nu_j(E)=\nu_j(E)=\hat\nu_j(E\setminus Q_{i+1})$, while for $i>j$, $\hat\nu_i(E\setminus Q_{i+1})<\nu_i(E)=0$, hence $\hat\nu_i(E\setminus Q_{i+1})=0$.  

An informal explanation may be in order here.  A finite height $\SB$-measure with values in levels $0,1,2\ldots h$ is a measure with several magnifications.  Using no microscope, we are at height h, and we see only sets with a very dense measure.    Putting our space under the microscope with a higher magnification, we see sets with less dense measure as having positive measure, sets that had measure 0 when viewed with the naked eye.  The microscope is now adjusted to see sets with positive measure at height $h- 1$.    Increasing the magnification of the microscope one notch, we once again get a much closer view and see even thinner sets.  The finite height measure is like a smart microscope:  When it looks at a set, it immediately chooses the right magnification to see the set as having measure $> 0$ at a certain height.

Clearly a real locally finite measure takes finite values on compact sets.  Assume now that $X$ is compact and that $\nu$ is a finite height locally finite $\SB$-measure, with associated sequences $\nu_j$ and $\hat \nu_j$.   This means $\hat\nu_j$ takes infinite values only on non-compact sets.    Then if $E$ is a Borel set with  $\bar E\subset  Q_j\setminus Q_{j+1}$, it follows that $\hat\nu_j(E)\le \hat\nu_j(\bar E)<\infty$ because $\bar E$ is compact.  This means
 $\hat \nu_j$ takes infinite values only on Borel sets $E\subset Q_j\setminus Q_{j+1}$ which have closure points in $Q_{j+1}$.  Condition (iv) is a kind of continuity condition.  It says that $\nu(E)=(j,\infty)$ only if $\bar E\cap Q_{j+1}\ne \emptyset$, i.e. only if it is ``close to"  $Q_{j+1}$, the set where $\nu$ has level $j+1$.   
  
 If one of the measures $\hat\nu_j$, $0\le j\le h$, is trivial, then in some sense the finite height measure is equivalent to another one, in which there are no trivial $\hat\nu_j$ for $0\le j\le h$.   In this case we could decrease by 1 the level of every $\hat\nu_i$ for $i>j$.   
 
 \begin{defn}\label{LevelAdjustment}  Suppose $\nu$ is a finite height $\SB$-measure on $X$.  
 Suppose $\hat \nu_j$ is trivial.  Then $\nu$ is {\it equivalent via level alignment} to the measure $\mu$, where $\mu$ is obtained from $\nu$ by shifting some levels down:  $\hat \mu_i=\hat\nu_i$ for $i<j$;  $\hat \mu_i=\hat\nu_{i+1}$ for $i\ge j$.
\end{defn}

We will use finite height $\SB$-measures in a very simple setting where $X$ is a finite union of closed interval transversals for a lamination in a surface $S$.  

\begin{defns} \label{FiniteDepthLam2} Let $S$ be a surface with cusps satisfying $\chi_g(S)<0$.  We will usually assume that we have chosen a hyperbolic structure for $S$.  An {\it essential lamination} in $S$ is a lamination which is isotopic to a geodesic lamination containing no boundary components of $S$.  A {\it finite height essential lamination} in a surface $S$ with cusps is an essential lamination with an invariant transverse finite height $\SB$-measure $\nu$, which assigns an element of $\SB$ to each closed interval transversal  for $L$.   To say that the transverse measure is invariant means that isotoping the transversal through transversals to $L$, using an isotopy  leaving endpoints of the transversal in the same leaf or complementary component of $L$ at all times, yields the same measure on the new transversal.  To say that the transverse measure is finite height means that the $\SB$-measure is a finite height measure on any transversal (or finite union of transversals) in the sense of Definition \ref{FiniteDepthMeasureDef} and the heights on different transversals have a uniform bound.  
A {\it finite height essential measured lamination} in a surface $S$ is a finite height lamination, with a locally finite invariant $\SB$-measure on transversals.  \end{defns}

\begin{lemma} \label{DefnsLemma}  The definitions of finite height lamination given in Definition \ref{FDepthDef} and Definition \ref{FiniteDepthLam2} are equivalent.  Also the definitions of finite height measured lamination given in Definition \ref{FiniteDepthLam} and Definition \ref{FiniteDepthLam2} are equivalent.
\end{lemma}

\begin{proof}  Suppose first that $L$ is a lamination in $S$  with an invariant transverse $\SB$-measure.   We choose finitely many closed transversals $T_i$ with the property that every leaf is intersected by at least one of the transversals, and we let $X$ denote the union of these transversals.  Then the transverse measure on $L$ yields an $\SB$-measure $\nu$ on $X$, which must, of course, satisfy certain compatibility conditions coming from the invariance of the transverse measure on $L$.   We may assume that the measures $\nu_i$ associated to the measure $\nu$ on $X$  are $\nu_{0},\nu_{1}, \ldots \nu_{h-1},\nu_h$.   We will also use $\nu$ to denote the transverse measure on $L$,  which is determined by the transverse measure on $X$.   Let $L_h$ be the union of leaves $\ell\subset L$ with the property that every open transversal $T$ through $\ell$ satisfies $\level(\nu(T))=h$.   The  $\SB$-measure $\nu$ on $X$ yields the associated $\nu_h$ on $X$ as in the above discussion.  Then $\nu_h$ is a  transverse $\bar\reals$-measure $\nu_h$ for $L_h$.
We claim $L_h$ is closed, for if $\ell$ is a leaf in the closure of $L_h$, then any open transversal $T$ through $\ell$ must intersect $L_h$, so $\level(\nu(T))=h$ and $\ell\subset L_h$.  Let $Y_h$ be the union of leaves $\ell$ of $L_h$ with the property that every open transversal $T$ through $\ell$ has infinite measure: $\real(\nu(T))=\infty$.  Then $Y_h$ is closed by the same reasoning.  Further $X_h=L_h\setminus Y_h$ is closed in $S\setminus Y_h$.

Inductively, having defined $L_h,L_{h-1}\ldots L_{h-j+1}$; $Y_h,Y_{h-1}\ldots Y_{h-j+1}; X_h,X_{h-1},\ldots, X_{h-j+1}$; let $L_{h-j}$ be the union of leaves $\ell\subset L$ which are not in $L_i$ for $i>h-j$ and with the property that for every open transversal $T$ through $\ell$ disjoint from $L_i$ for $i>h-j$, we have $\level(\nu(T))=h-j$.  Letting $\hat S_{h-j}$ denote $S-\{L_h\cup L_{h-1} \cup\cdots \cup L_{h-j+1}\}$, we claim that $L_{h-j}$ is closed in $ \hat S_{h-j}$.  The proof is the same as for $L_h$:  If $\ell$ is a leaf in the closure of $L_{h-j}$ in $\hat S_{h-j}$ and $T$ is an open transversal in $\hat S_{h-j}$ intersecting $\ell$, then $T$ intersects leaves of $L_{h-j}$, which implies $\level(\nu(T))=h-j$ and $\ell\subset L_{h-j}$.  Let $Y_{h-j}$ denote the union of leaves $\ell$ in $L-\{L_h\cup L_{h-1} \cup\cdots \cup L_{h-j+1}\}$ with the property that any open transversal $T$ of $\ell$ in $\hat S_{h-j}$ satisfies $\real(\nu(T))=\infty$.  Then as before, we can show $Y_{h-j}$ is closed in $\hat S_{h-j}$.  

We have shown that if $L$ is a finite height essential lamination according to Definition \ref{FiniteDepthLam2}, then it is a finite height lamination according to Definition \ref{FDepthDef}.  

Now assuming $L$ is a finite height lamination according to Definition \ref{FDepthDef}, we will show it is a finite height lamination according to Definition \ref{FiniteDepthLam2}.  Given $L$ as a union $L=\bigcup_{j=0}^hL_j$ where $L_j=X_j\cup Y_j$, with $X_j$ having a real invariant transverse measure $\mu_j$ we define an invariant transverse $\SB$-measure $\nu$.  If $T$ is a transversal, let $\level(\nu(T))$ be the greatest $j$ such that the interior of $T$ intersects $L_j$.  Then $\intr(T)\subset \hat S_j$, otherwise the interior of the transversal would intersect leaves at a higher level.  If $T$ intersects $Y_j$ or an endpoint of $T$ intersects $L_{j+1}$, we assign $\real(\nu(T))=\infty$, otherwise we assign $\real(\nu(T))=\mu_j(T)$.  Of course we assign $\level(\nu(T))=j$.  Now we can easily verify that $\nu$ is a transverse finite height $\SB$-measure.

It remains to prove the statement about finite height measured laminations.  In this case, for one implication, we assume we have a transverse invariant locally finite $\SB$-measure $\nu$ on $L$.    Following the previous argument, we obtain $L_j$, $Y_j$, and $X_j$ as before, with an $\reals$-measure on $X_j$,  but $Y_j$ is empty.  Conversely, starting with a finite height measured lamination $L=L_0\cup L_1\cup \cdots L_h$, with transverse $\reals$-measure $\mu_j$ on $L_j$, we define the transverse $\SB$-measure $\nu$ as before:  If $T$ is a transversal, let $\level(\nu(T))$ be the highest level $j$ such that the interior of $T$ intersects $L_j$, we assign $\nu(T)=(j,\mu_j(T))$.  Then it is easy to verify that $\nu$ is a transverse finite height locally finite $\SB$-measure.  \end{proof}
 
 Now we can give a proof of Proposition \ref{PlanteProp}, which can also be restated in terms of transverse $\SB$-measures.
 
 \begin{proof}[Proof of Proposition \ref{PlanteProp}]  We use an old argument due to J. Plante, \cite{JP:PlanteLimit}.   Choose any leaf $\ell$ of $L$.  We choose an exhaustion of $\ell$ by an nested sequence of subarcs $I_n$, $\cup_nI_n=\ell$.  We can define a transverse measure $\mu_n$ for $L$ which just counts the number of intersections of a transversal $T$ with $I_n$.   The transverse measures $\mu_n$ are not invariant.  We can assume $L$ is carried by a train track $\tau$ and normalize $\mu_n$ such that the maximum weight induced on a segment of $\tau$ is $1$.  Then $\mu_n$ (viewed as a weight vector on $\tau$) has a convergent subsequence, converging to a weight vector on $\tau$, which determines a transverse measure $\nu_0$ with support $L_0$, say, a sublamination of $L$.    Now we consider the lamination $L\setminus L_0$ in $\hat S_1$, the completion of $S\setminus L_0$, and we choose a leaf $\ell$ of $L\setminus L_0$ and construct a transverse measure using Plante's argument again to obtain a transverse measure $\nu_1$ with support $L_1$.  We proceed inductively to obtain the finite height structure.  The process stops with $L_h$, because at each step of the argument at least one more segment of $\tau$ is assigned a non-zero weight at some level.  Finally, we observe that our indexing needs to be reversed, so $L_0$ becomes $L_h$ and vice versa.
 \end{proof}
 
 \begin{remark}  Proposition \ref{PlanteProp} can be restated as follows:  Every geodesic lamination admits a locally finite transverse $\SB$-measure.
 \end{remark}


Before we can define a set $\F(S)$ and $\FC(S)$ of finite height laminations in $S$, we must describe equivalence relations on finite height laminations whose classes are elements of the sets $\F(S)$ and $\FC(S)$.

\begin{defn}\label{ContiguousDef} If $L_j=\emptyset$, the finite height lamination given as \\$L=((L_0,\nu_0),(L_1,\nu_1),\ldots,(L_h,\nu_h))$, with $\nu_i$ having full support on $L_i$, is {\it equivalent by alignment} to the lamination $L=((L_0,\nu_0),(L_1,\nu_1),\ldots,(L_{j-1},\nu_{j-1}),(L_{j+1},\nu_{j+1}),\ldots, (L_h,\nu_h))$, the latter having height $h-1$.    Each such equivalence class is represented by a {\it proximal lamination}, i.e. a lamination $L=((L_0,\nu_0),(L_1,\nu_1),\ldots,(L_h,\nu_h))$  for which every $L_i\ne \emptyset$, $0\le i\le h$.  The set of proximal finite height essential laminations is called $\F(S)$.
 \hip
 {\it Level adjustment:} Suppose $L\in \F(S)$ is given as above, presented as\\  $((L_0,\nu_0),(L_1,\nu_1),\ldots,(L_h,\nu_h))$.   Suppose  $j<h$ and $L_j=L_j'\cup L_j''$, where $(L_j',\nu')$ and $(L_j'',\nu'')$ are laminations in $\hat S_j$, and $(L_{j+1}\cup L_j'',\nu_{j+1}')$ is a $\bar\reals$-measured lamination in $\hat S_{j+1}$ (or in $S$ if $j+1=h$), with $\nu_{j+1}'=\nu_{j+1}+\nu_j''$ on $L_{j+1}\cup L_j''$.   Then we obtain a modified finite height lamination $L'=((L_0,\nu_0),(L_1,\nu_1),\ldots, (L_{j},\nu_{j}),(L_{j+1}\cup L_j'',\nu_{j+1}+\nu_j''),(L_{j+2},\nu_{j+2}), \ldots (L_h,\nu_h))$.  \hop

If $L=((L_0,\nu_0),(L_1,\nu_1),\ldots,(L_d,\nu_d))$ has the property that no level adjustment modifications or level alignments are possible, then we say $L$ is a {\it contiguous finite height lamination} and belongs to $\FC(S)$.  \end{defn}

Informally, a proximal (contiguous) finite height lamination is one with the property that ``there are no unnecessary gaps between levels," where ``gaps" are defined differently in the two cases.  Also informally, a level alignment moves an $L_j$ down one level, while a level adjustment moves a component of $L_j$ up one level.  

\begin{proposition}  Suppose a finite height lamination $L$ is presented as \\
 $((L_0,\nu_0),(L_1,\nu_1),\ldots,(L_d,\nu_d))$, with $L_j=X_j\cup Y_j$.    If the measured part $X_j$ of $L_j$ is empty for all $j$, then $L$ is equivalent by level alignment and level adjustment to a lamination $L'$ of the form $(L'_0,\nu'_0)$ with $X'_0=\emptyset$ and $\nu_0'$ everywhere locally infinite.  These laminations are exactly the laminations of $\GL(S)$ described in the introduction.
\end{proposition}

The proof of the proposition is immediate.

The proposition shows that without non-trivial measured sublaminations $X_j$ at some levels, finite height contiguous laminations as we defined them have total height 0.  That is not the case for proximal laminations.

We finish this section with an alternative description of the parameter space $\SB$.  The alternative description will be useful in analyzing finite height laminations.  We let $\II$ denote the interval in $\bar\reals$.  We will realize $\SB$ as a subspace of $\II^\nat$ with the product topology.  

\begin{defn}  Let $\SBB$ denote the subspace of $\II^\nat$ with the product topology consisting of elements $(x_i)_{i\in\nat}$ such that for some distinguished $k\in \nat$, $x_i=\infty$ for $i<k$, $x_i=0$ for $i>k$, and $0<x_k\le \infty$.  In addition $\SBB$ contains the 0 element $(x_i)$ with $x_i=0$ for all $i$.  \end{defn}

\begin{remarks} (1) We are using slightly different symbols $\SB$ and $\SBB$ for the two versions of our parameter space.  (2) For an element $x=(x_i)\in \SBB$, the distinguished index $k$ is the last $i$ such that $x_i>0$, so $k$ is determined by $x$.
\end{remarks}

\begin{proposition} \label{LHomeoProp} The spaces $\SB$ and $\SBB$ are homeomorphic via a homeomorphism $\psi:\SB\to\SBB$ defined by 
$$\psi((k,a))=(\infty,\infty\ldots,\infty,a,0,0,0\ldots),\ \psi(0)=(0,0,\ldots,0,0,\ldots),$$ where the entry $a$, $0<a\le\infty$, is in the $k$ factor of $[0,\infty]^\nat$.
\end{proposition}

\begin{proof}  Clearly $\psi $ is injective and onto.  We have standard bases for both spaces:  In $\SB$ we have a basis of  intervals in the order topology; in $\SBB$ we have a basis consisting of basis neighborhoods in the product topology for $[0,\infty]^\nat$ intersected with $\SBB$. We will show that sufficiently small basis neighborhoods of points in both topologies are the  same.  

In $\SB$ a small basis neighborhood of a point $(k,a)$ where $a\ne \infty$ takes the form $((k,c),(k,d))$ with $c<a<d$.  Now $\psi((k,a))=(\infty,\infty\ldots,\infty,a,0,0,0\ldots)$ with distinguished entry at $k$.  A small basis neighborhood of this point in the product topology intersected with $\SBB$ takes the form $\{\infty\}\times\{\infty\}\cdots\times\{\infty\}\times\{\infty\}\times (c,d)\times \{0\}\times \{0\}\times\cdots$ which is the image of the basis neighborhood of $(k,a)$ in $\SB$.

In $\SB$ a small basis neighborhood of $(k,\infty)$ takes the form of the interval $((k,c),(k+1,d))$, with $c$, $d$ finite.  The image is 
$$\psi\Big(\big((k,c),(k+1,d)\big)\Big)=\psi\Big(\big((k,c),(k,\infty)\big]\Big)\cup \psi\Big(\big((k,\infty),(k+1,d)\big)\Big)=$$
\vskip -0.3in
$$\left(\{\infty\}\times\{\infty\}\times \cdots\times\{\infty\}\times (c,\infty]\times \{0\}\times \{0\}\times\cdots\right)\cup\left(\{\infty\}\times\cdots\times\{\infty\}\times (0,d)\times  \{0\}\times\cdots\right)=$$
\vskip -0.3in
$$\left(\{\infty\}\times\{\infty\}\times \cdots\times\{\infty\}\times (c,\infty]\times \{0\}\times \{0\}\times\cdots\right)\cup\left(\{\infty\}\times\cdots\times\{\infty\}\times [0,d)\times  \{0\}\times\cdots\right),$$

\noindent which is a typical small basis neighborhood in $\SBB$ of $( \infty,\infty,\ldots, \infty,0, 0,0,\ldots)$ with the last $\infty$ in the $k$ entry.  

A typical small basis neighborhood of the point $0\in \SB$ has the form $\big[0, (0,c)\big)$ with $c\ne\infty$. The image in 
$\SBB$ has the form $[0,c)\times \{0\}\times \{0\}\times\cdots$, which is also a typical small basis neighborhood of $0$ in $\SBB$.  \end{proof}

Evidently we can now transfer the algebraic structure of $\SB$ to $\SBB$.

\section {The topology}

In order to describe a topology for $\F(S)$, $\FC(S)$, $\PF(S)$, and $\PFC(S)$, it is convenient first to understand a topology for ``the set of finite height laminations in one level," i.e., laminations of the form $L_0=X_0\cup Y_0$ in the surface-with-cusps $S$, where $X_0$ has an everywhere locally finite transverse $\reals$-measure, and $Y_0$ has an everywhere locally infinite transverse measure.  The set of laminations at a different level are the same, except they exist in a different surface-with-cusps.  Thus $L_i=X_i\cup Y_i$ lives in the completion $S_i$ of $\displaystyle\hat S_i=S\setminus \bigcup_{j>i}L_i$.

\begin{defn}  A  {\it weakly measured essential lamination, or $\bar\reals$-measured lamination} in a surface-with-cusps $S$ is an essential lamination $L$ with a transverse measure of full support having values in $[0,\infty]\subset \bar\reals$, the extended reals.  (The transverse measure of a transversal should be positive whenever the transversal intersects the lamination.)  To say the lamination is essential means that it can be realized as a geodesic lamination with no leaves in $\bdry S$.  The set of weakly measured laminations in $S$ will be denoted $\WM(S)$.  The completion $\bar L=L\cup \bdry S$ can be regarded as a non-essential weakly measured lamination in $S$ if we assign an atomic infinite measure to each component of $\bdry S$, except closed curves approached by spiral leaves.  In particular, if $L$ is $\reals$-measured, then $\bar L$ is called a {\it geometric measured lamination in $S$}.
\end{defn}

We remark that if $L$ is $\reals$-measured in $S$, then $\bar L$  is called ``geometric" because $\DL$ approximates a hyperbolic structure on $\breve S$.

Recall that for a measured lamination in $S$, we require transverse measures to be finite on closed transversals in the interior of $S$, so this is the requirement we are discarding in the above definition, though we keep that requirement for geometric measures.   Thus if  $\M(S)$ denotes the set of measured laminations in $S$, we have $\M(S)\subset \WM(S)$.   As we have mentioned before, eliminating the requirement of local finiteness allows laminations which have no transverse structure.

For ordinary measured lamination space, the topology is obtained using parameters equal to intersection numbers of homotopy classes of closed curves in the surface with measured laminations.  To describe a topology for $\WM(S)$, after choosing a hyperbolic structure for $S$, we will use intersection numbers of embedded geodesics with weakly measured laminations, which intersection numbers have values in $[0,\infty]$.

For technical reasons we will use multiple copies of some of the geodesics.  For every closed geodesic $\gamma$ not in $\bdry S$, we will use five copies of $\gamma$.  The five copies arise as follows.
We can perturb $\gamma$ normally in two directions to get $\gamma'$ and $\gamma''$.  Each of these has a preferred normal, corresponding to the direction of perturbation.  We can think of $\gamma'$ as ``infinitesimally" perturbed, but sometimes we will have to think of it as actually perturbed.  If the closed curve $\gamma'$ is oriented so that the perturbation normal for $\gamma'$ followed by the orientation of $\gamma'$ yields the orientation of $S$, we obtain the oriented geodesic $\gamma_+'$.   Choosing the orientation for $\gamma'$ such that the perturbation normal followed by the orientation on $\gamma'$ yields the opposite of the orientation on $S$, we obtain $\gamma_-'$.  In the same way we obtain $\gamma''_+$ and $\gamma''_-$.  For each closed geodesic in $\bdry S$ we also include oriented copies:   If $\gamma$ is a closed curve in $\bdry S$, we use only $\gamma'$, where $\gamma'$ is $\gamma$ perturbed to the interior of $S$. The orientations of $\gamma'_+$ and $\gamma'_-$ are determined as before.

\begin{defn} Suppose $S$ is a surface-with-cusps. Then $\G$ denotes a set of geodesics in $S$, including components of $\bdry S$,  with some geodesics represented multiple times as follows:

\begin{tightenum}
\item For every closed geodesic $\gamma$ in $\bdry S$, we include two oriented perturbed geodesics $\gamma'_+$, and $\gamma'_-$,  perturbed towards the interior of $S$ with orientations as described above.
\item For every closed geodesic $\gamma$ not in $\bdry S$ we include five copies of $\gamma$ denoted $\gamma$, $\gamma'_+$, $\gamma'_-$,   $\gamma''_+$, and $\gamma''_-$ with $\gamma'_+$, $\gamma'_-$, $\gamma''_+$, and $\gamma''_-$ oriented as described above.
\item Every non-closed geodesic $\gamma$ is included only once without orientation.
\end{tightenum}

We say that $\gamma'_+$, $\gamma'_-$,   $\gamma''_+$, and $\gamma''_-$ are {\it oriented variants} of the closed curve $\gamma$. \end{defn}

The set of weakly measured laminations $\WM(S)$ will be topologized as a subspace of $[0,\infty]^\G$ by mapping a weakly measured lamination $(L,\mu)$ to the point $\I(L,\mu)=(i_\theta(L,\mu))_{\theta\in \G}$, where $i_\theta(L,\mu)$ is an intersection number we will now define.  The value of  $i_\theta(L,\mu)$ is calculated differently, depending on whether or not $\theta$ is an oriented variant. 

\begin{defn} If $\theta=\gamma$ is not an oriented variant,  $i_\theta(L,\mu)$ denotes the total  transverse measure (or length) of $\theta=\gamma$ as measured by $\mu$.   In particular, if $\theta=\gamma$ is not a variant and $\gamma$ is a leaf of $L$, we make the convention that $i_\theta(L)=0$.  Next, we define $i_\theta(L)$ for variants $\gamma'_+$, $\gamma'_-$,   $\gamma''_+$, or $\gamma''_-$ of a closed geodesic $\gamma$ in the interior of $S$.  If $\gamma$ is not a leaf of $L$, we define $i_\theta(L,\mu)=0$ for $\theta$ equal to any of the variants $\gamma'_+$, $\gamma'_-$,   $\gamma''_+$, or $\gamma''_-$.  If $\gamma$ is a closed leaf of $L$ and $\theta$ is a variant of  $\gamma$, so $\theta$ is $\gamma'_+$, say, with an orientation, we define $i_{\theta}(L)$ to be the measure of leaves spiraling to $\gamma$ from the side to which $\gamma$ is perturbed to get $\gamma'_+$, provided the orientation of $\theta$ is related to the sense of spiraling as shown in Figure \ref{Sign}.  In the same figure, if $\theta=\gamma'_-$, $i_\theta(L)=0$.  Still considering $\theta=\gamma'_-$, the number $i_\theta(L)$ detects the measure of leaves of $L$ spiraling in the opposite sense from that in the figure.  If no leaves spiral to $\gamma$, $i_\theta(L)=0$ for all variants of $\gamma$.  Similarly, if $\theta$ is an oriented perturbed closed curve $\gamma$ of $\bdry S$,  we define $i_{\theta}(L)$ to be the measure of leaves spiraling to $\gamma$, provided the sense of spiraling agrees with the orientation on $\theta$ as shown in the figure.  Again, if the sense is opposite, $i_\theta(L)=0$.

 If $\rho\in \G$ is a non-oriented geodesic spiraling to a close curve $\gamma\in \G$ on the side to which $\gamma'$ is perturbed, then we say $\theta=\gamma'_+$ {\it intersects} $\rho$ provided the sense of the spiraling agrees with the orientation as shown in Figure \ref{Sign}, otherwise it does  {\it not intersect}.   If $\rho, \theta\in \G$ are both non-oriented, they {\it  intersect} if the underly geodesics $\beta,\gamma$ corresponding to $\rho$ and $\theta$ intersect transversely.  In all other cases, $\rho$ and $\theta$ do not intersect.   More precisely, we have a function $\K:\G\times \G\to \{0,\infty\}$ whose value $\K(\rho,\theta)$ is $\infty$ if $\rho$ and $\theta$ intersect, $0$ otherwise.   Assuming we have chosen a hyperbolic structure for the surface and we have made $\beta,\gamma$ geodesic, we define the intersection function $\K$ as follows:
 
 \hop\noindent If $\rho$ and $\theta$ are non-oriented:
 
  $$\K(\rho,\theta)
=\begin{cases}\infty & \mbox{if }\beta,\gamma \mbox{ intersect and do not coincide }\\
                     0& \mbox{ if } \beta,\gamma \mbox{ coincide or are disjoint}
                                                                \end{cases}
$$
\noindent  If $\rho$ and $\theta$ are both oriented $\K(\rho,\theta)=0$.
\hop \noindent  If $\rho$ is not oriented, $\theta$ is oriented:

$$\K(\rho,\theta)
=\begin{cases}
                     \infty& \mbox{ if }\beta\mbox{ spirals to }\gamma\mbox{, with }\theta \mbox{ oriented as in Figure \ref{Sign}}\\
                     0& \mbox{ otherwise. }
                                            \end{cases}
$$
\end{defn}

The intersection function $\K$ is determined by the surface, regardless of the choice of hyperbolic structure for $S$.  The only purpose of the oriented variants is to detect (the measure of) leaves spiraling to a closed curve.

\begin{figure}[H]
\centering
\scalebox{1}{\includegraphics{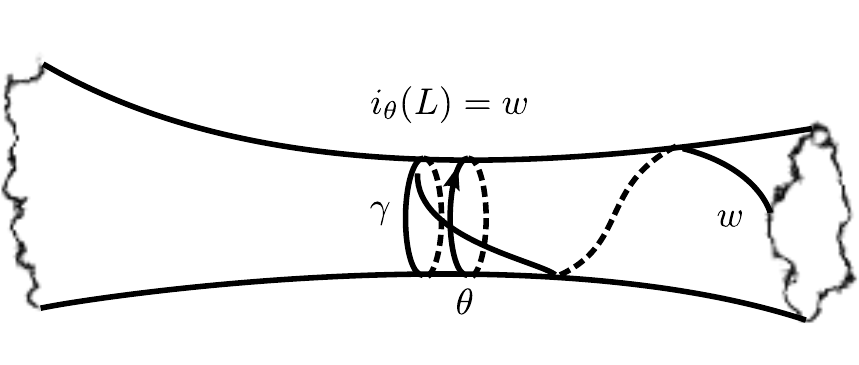}}
\caption{\small Intersection with variant closed curve.}
\label{Sign}
\end{figure}

Whether or not $\theta$ is an oriented variant,  $i_\theta(L,\mu)$ can be infinite, which is why we allow values in $[0,\infty]\subset \bar\reals$.  
A leaf spiraling towards $\gamma$, with finite atomic transverse measure, determines a cohomology class in $z\in H^1(\gamma,\reals)$.  Letting $\theta=\gamma'_+$ be the oriented curve, viewed as an oriented closed path in $\gamma$, then

$$i_{\gamma'_+}(L,\mu)=
\begin{cases}z(\theta)& \mbox{if } z(\theta)> 0 \\
                   0 & \mbox{otherwise}
                                            \end{cases}
$$

$$i_{\gamma'_-}(L,\mu)=
\begin{cases}-z(\theta)& \mbox{if } z(\theta)< 0 \\
                   0 & \mbox{otherwise}
                                            \end{cases}
$$

The above formulas show that we have split one $\bar\reals$-parameter $z$, with positive and negative values, which gives the measure and sense of leaves spiraling to a closed curve of $\bdry S$, into two  $\II$-parameters $i_{\gamma'_+}$ and $i_{\gamma'_-}$.
  
  We need to show that the map $\I$ is injective, and then we will redefine $\WM(S)$ as the topological subspace $\I(\WM(S))\subset [0,\infty]^\G$.

\begin{lemma}\label{InjectiveLem}  The map $(L,\mu)\mapsto \I(L,\mu)=(i_\gamma(L,\mu))_{\gamma\in \G}$ from the set  $\WM(S)$ of weakly measured geodesic laminations to $\II^\G$ is injective. 
\end{lemma}

\begin{proof}  The lemma is proved for ordinary measured laminations in \cite{UO:Pair}.   Recall that even in that setting, a transversal of a measured lamination $L$ in $S$ with one endpoint in a closed curve $\bdry S$ can have infinite transverse measure if the lamination has leaves spiraling to the closed curve.  So even for measured laminations, we must allow infinite transverse measures, but they are finite in the interior of $S$.  We now allow laminations which have transverse measures which need not be locally finite even in the interior of $S$.  Given a weakly measured essential geodesic lamination $(L,\mu)$ in $S$, let $L_\infty$ denote the union of leaves $\ell$ with the property that every open transversal intersecting $\ell$ has infinite measure.  Then $L_\infty$ is clearly closed, therefore a sublamination of $L$. 

We claim that $L_\infty$ is determined as a geodesic lamination by the intersections $i_\gamma(L)$ for $\gamma\in \G$.   In fact, we claim $L_\infty$ is the union of geodesics $\gamma$ in the set 

$$G=\{\gamma\in\G|\text{$\gamma$ is non-oriented and not a component of $\bdry S$, }$$\vskip -0.3in $$i_\gamma(L)=0\text{ and }\nexists \rho\in \G \text{ with } 0\le  i_{\rho}(L)<\infty \text{ and } \rho \text{ intersects }\gamma\}.$$  

\noindent  Note that $G$ is actually determined by the intersection numbers $i_\theta$, since the intersection function $\K$, which tells us when two elements of $\G$ intersect, is given with the surface $S$. 

 To prove the claim, first suppose $\gamma$ is a geodesic in $L_\infty$.  Then $i_\gamma(L)=0$.  Further, if $\rho\in \G$ and $\rho$ intersects $\gamma$, then $i_\rho(L)=\infty$, even if $\rho$ is an oriented closed curve.  This shows $\gamma\in G$.  

Conversely, suppose $\gamma\in G$.  Then $\gamma$ is a non-oriented geodesic, and since $i_\gamma(L)=0$, we know that 
\begin{tightenum}
\item $\gamma$ is a leaf of $L_\infty$, or 
\item $\gamma$ is a geodesic in the completion $F$ of the complement $S\setminus L_\infty$, and $\gamma$ is not a component of $\bdry S$. 
\end{tightenum}
We will derive a contradiction if (ii) holds to finish the proof of our claim.  $F$ is a surface with cusps and $L\setminus L_\infty$ is a measured lamination in $F$, whose measure has full support.     Given $\rho\in \G$, $\rho\subset F$, $i_\rho(L)=i_\rho(L\setminus L_\infty)<\infty$ unless $\rho$ is bi-infinite and approaches a closed curve of $\bdry F$ in the opposite sense that the lamination $L$ does.  Since the component of $\check F$ of $F$ containing $\gamma$ is not a product, it only remains to show that for any essential $\gamma$ in any such $\check F$, we can find an essential $\rho\in \G$ intersecting it with specified senses of spiraling at ends which approach closed boundary curves of $\bdry \hat F$ so that $ 0\le  i_{\rho}(L)<\infty$, giving a contradiction.   This is a relatively easy exercise, but we do need oriented elements of $\G$.  The extreme example where it is most difficult to find a $\rho$ fitting our needs is a component $\check F$ which is an annulus with one cusp on one boundary and with one closed boundary component, a closed curve of $\delta$.  In this case $\gamma$ must be a geodesic with one end in the cusp and the other end spiraling to the closed curve.  One of the oriented variants $\theta$ of the closed curve in $\bdry\check F$ intersects $\gamma$ and has finite intersection with it. 
  
Now we must show that the remainder of $L$ is determined by intersections.  But the remainder of $L$ is a measured lamination in the complement of $L_\infty$, i.e. in the same $F$ as above, so it is determined by intersections with ``geodesics" $\H$ in $F$, by the result in \cite{UO:Pair}.   Here $ \H$ includes oriented boundary closed curves, which are treated like closed curves perturbed to the interior of $F$ with non-trivial intersections with spiral geodesics approaching the closed curve.   Clearly we can regard $\H$ as a subset of $\G$.  \end{proof}

\begin{remark}  There are certain unoriented elements of $\G$ which we call {\it spiral geodesics}.  These are geodesics with at least one end spiraling to a closed curve.  If we examine the above proof, we observe that we do need these elements of $\G$ to ensure that intersection numbers determine $L_\infty$.  However, if $L_\infty$ is empty, i.e. when we are dealing with measured laminations in the sense of \cite{UO:Pair},  the lamination is determined by the other elements of $\G$.
\end{remark}

\begin{remark}  Intersection numbers do not distinguish between $L$ and $\bar L$, the latter being a lamination which includes $\bdry S$.
\end{remark}

It may be worth pointing out the following corollary of Lemma \ref{InjectiveLem}.

\begin{corollary}\label{GeodesicInjectiveCor}  The map $\I:\GL(S)\to \{0,\infty\}^\G$ is injective.
\end{corollary}

\begin{defn} The space $\WM(S)$ is the image of the set $\WM(S)$ under the map $\I:\WM(S)\to[0,\infty]^\G$ with the subspace topology.  The space $\PWM(S)$ is the quotient of $\WM(S)$ under the projectivization equivalence in $[0,\infty]^\G\setminus\{0\}$.  Namely if $v,w\in [0,\infty]^\G$ are non-zero then $v\sim w$ if $v=\lambda w$ for some $\lambda\in \reals$, $0<\lambda<\infty$, with the convention that $\lambda\infty=\infty$.
\end{defn}

\begin{remark}  In the space $\WM(S)$ of weakly measured laminations, there is a subspace $\GL(S)$ of laminations such that $L=L_\infty$.  Projectivization has no effect on $\GL(S)\subset \WM(S)$ since equivalence classes are singletons.  Thus $\GL(S)$ is a subspace of $\PWM(S)$.   \end{remark}

\begin{proposition}  The space $\GL(S)$ is a subspace of $\WM(S)$.  The space $\PWM(S)$ is non-Hausdorff, with any neighborhood in $\PWM(S)$ of a lamination $K\in \GL(S)$ containing all projective classes of $[L]\in \PWM(S)$ such that the underlying lamination for $L$ is $K$.  $\PWM(S)$ contains $\GL(S)$ and $\PM(S)$ as subspaces.
\hip
\noindent Suppose $K\in \GL(S)\subset \WM(S)$ is a fixed lamination.  Then the subspace $\M_K(S)$ of laminations $L\in \WM(S)$ such that $L_\infty=K$ is homeomorphic to $\M(F)$, where $F$ is the completion of $S\setminus K$.  The projective image $\PM_K(S)$ of $\M_K(S)$ is a subspace of $\PWM(S)$ homeomorphic to $\PM(F)$.
\end{proposition}

\begin{proof} The main issue here is showing that the quotient of a subspace is the corresponding subspace of the quotient.  This is true if a subspace is either open or closed.  It is easy to check that $\GL(S)\subset \WM(S)$ is closed:  If $K\in \WM(S)\setminus \GL(S)$, then for some $\theta$, $i_\theta(K)<\infty$ and $\{L\in \WM(S): i_\theta(L)<i_\theta(K)+ 1\}$ is a basis neighborhood of $K$ disjoint from $\GL(S)$.  We have already observed that projective equivalence classes in $\GL(S)$ are singletons, so the quotient map $q: \WM(S)\to \PWM(S)$ restricted to $\GL(S)$ is also a quotient map, hence a homeomorphism.  It is also easy to check that $\M(S)$ is open in $\WM(S)$, since $\M(S)$ is the subspace of $[0,\infty]^\G$ with all coordinates finite, except coordinates corresponding to spiral geodesics described in the above remark. (In \cite{UO:Pair} we use intersection numbers with curves of $\H\subset \G$, which excludes the spiral geodesics, because they are not needed to determine measured laminations.)  Hence the quotient map $q: \WM(S)\to \PWM(S)$ restricted to $\M(S)\subset \WM(S)$ is a quotient map yielding $\PM(S)$.

To see that a neighborhood of $K\in \GL(S)$ contains any projective class $[L]$ such that the underlying lamination of $L$ is $K$, observe that a basis neighborhood of $K$ in $\WM(S)\subset [0,\infty]^\G$ restricts finitely many coordinates $i_{\theta}(L)$ to lie in a neighborhoods of $\infty\in [0,\infty]$.  By multiplying the measure on $L$ by a sufficiently large $\lambda$, we find a representative of the projective class $[L]$ in the basis neighborhood.

Now we consider the second statement of the proposition.  Given a lamination $L\in \M_K(S)$ clearly $i_\theta(L)$ is finite for any $\theta$ with underlying geodesic in $F$ (again excluding geodesics with one end spiraling to a closed curve of $\bdry F$).  By this we mean the underlying geodesic of $\theta$ lies in $S\setminus K$ or if it is an oriented variant, it is perturbed into $S\setminus K$.  If $\H\subset \G$ is the set of geodesics in $F$, excluding geodesics spiraling to closed boundary curves, then $L$ is determined by coordinates in $[0,\infty)^\H\subset [0,\infty]^\G$, so we have an open subspace homeomorphic to $\M(F)$.  Projectivizing, we get a subspace of $\PWM(S)$ homeomorphic to $\PM(F)$.
\end{proof}

Note that at other points of $\PWM(S)$ there can also be non-Hausdorff behavior.

The next task is to topologize the set $\F(S)$ of $\SB$-measured laminations.  To do this, we will again use the fact that the surface with cusps $S$ can be given a hyperbolic structure, since we assume $\chi_g(S)<0$.    We let $\G$ denote the set of geodesics in $S$, again including oriented variants.  

\begin{defn} Suppose $L\in \F(S)$, where $L$ is a proximal representative, with levels $0$ and $h$ represented, then for each $\theta$ in $\G$ we define the {\it intersection} $i_\theta(L)$ of $\theta$ with $L$ as $(j,t)\in \SB$ where $j$ is the largest $i$ such that $t=i_\theta(L_i)\ne 0$, when $L_i$ is regarded as a weakly measured lamination in $S_i$, the completion of $S\setminus \cup_{k>i}L_k$, and $\theta$ is regarded as an (oriented) curve in the same surface.  
If $\theta$ intersects {\it no}  leaves of $L$, or if $\theta$ is not oriented and coincides with a leaf of $L$  we define $i_\theta(L)=0\in \SB$.  
For every finite height essential lamination $(K,\nu)\in \F(S)$ we now have a function $\I(K,\nu):\G\to \SB$, namely ,  $\I(K,\nu)(\theta)=i_\theta((K,\nu))$, which can be regarded as a point in the product $\SB^\G$.  
\end{defn}

Clearly the description of $L$ in the definition above also gives a transverse $\SB$-measure for $L$, and $i_\theta(L)$ can be interpreted as an intersection in terms of the transverse $\SB$-measure on $L$, even when $\theta$ is an oriented variant.  Thus the alternate definition is as follows.  If $\gamma$ is the underlying geodesic for $\theta$, and $\theta$ is not oriented, $i_\theta(L)$ is the induced $\SB$-measure on $\gamma$.  If $\theta$ is oriented and the underlying geodesic is a closed curve $\gamma$ in $L$ then $\theta$ comes with a direction of normal perturbation and an orientation.  Then $i_\theta(L)$ is the $\SB$-measure of leaves of $L$ spiraling to $\gamma$ on the appropriate side of $\gamma$, provided the sense of spiraling agrees with the orientation of $\theta$.   Otherwise $i_\theta(L)=0$.

If we can show that $\F(S)$ maps injectively into $\SB^\G$, we can define a topology on $\FD (S)$ as a subspace of $\SB^\G$, and then we can view $\FC(S)$ as a subspace.  This approach does indeed work.  For the proof, we define $\FD_k(S)$ as the set of $\SB$-measured laminations of height $\le k$.  

\begin{proposition}  \label{InjectiveProp} The map $\I:\F(S)\to \SB^\G$ defined by $\I(L)=(i_\gamma(L))_{\gamma\in\G}$ is an injection.
\end{proposition}

\begin{proof}  The hard work has already been done in Proposition 4.2 of \cite{UO:Pair} and Lemma \ref{InjectiveLem}.  We have injectivity for the map sending a weakly measured lamination $L$ in  $S$ 
to a point $(i_\gamma(L))_{\gamma\in\G}\in[0,\infty]^\G$.

We prove by induction that a finite height lamination $L=L_0\cup L_1\cup\cdots \cup L_h$ is determined by its intersection numbers with curves of $\G$.   If $h=0$, $L$ has height 0, and $L$ is a weakly measured lamination in $S$, so it is determined by its intersection numbers.   Now suppose we have shown that height $h-1$ laminations are determined by intersection numbers with values in $\SB$.   Given a height $h$ lamination $L=L_0\cup L_1\cup\cdots \cup L_h$, we then know that the height $h-1$ lamination $L_0\cup\cdots \cup L_{h-1}$ in the completion $F$ of $S\setminus L_h$ is determined by intersection numbers with curves of $\G$ in $F$.  Further, we know that the weakly measured lamination $L_h$ in  $S$ is determined by intersection numbers with elements of $\G$. This shows that the entire lamination is determined by $\SB$-intersection numbers.  
\end{proof} 

\begin{remark} Again we remark that if $\bar L$ denotes $\bar L_0\cup \bar L_1\cup\cdots \cup \bar L_h$, the latter is not distinguished from $L$ by intersection numbers, but the modification replacing $L$ by $\bar L$ is done in a standard way, so $\bar L$'s are distinguished from each other.  
\end{remark}

\begin{defns} \label{SpacesDefns} The topological space of finite height laminations $\F(S)$ is the image $\I(\F(S))$ in $\SB^\G$ with the subspace topology.    The {\it subspace space of contiguous finite height laminations}, $\FC(S)$ is the subspace of contiguous finite height laminations, which were defined in Definition \ref{ContiguousDef}.  The {\it space of finite height measured laminations}, $\FM(S)$ is the subspace of finite height measured laminations in $\F(S)$.  The {\it  subspace of contiguous finite height measured laminations}, $\FCM(S)$ is the subspace in $\FC(S)$ of laminations which are both finite height measured and contiguous.

Two points in $u,v\in \SB^\G$ are {\it projectively equivalent} if there exists $\lambda\in \reals$, $0<\lambda<\infty$ such that  $u=\lambda v$.   Thus we obtain a quotient under this equivalence relation of the injective image $\I( \F(S))\setminus \{0\}$, which we call $\PF(S)$.  There is a subspace $\PFM(S)$, the {\it projective finite height measured lamination space}.   $\PFC(S)$, the {\it contiguous projective finite height lamination space} is the projective quotient of $\FC(S)$; and $\PFCM(S)$, the {\it contiguous projective finite height measured lamination space} is the projective quotient of $\FCM(S)$.
\end{defns}

\begin{remark}  There is a lemma buried in the above definition.   Again, it is useful to know, for example, that $\PFM(S)$ is actually a subspace of $\PF(S)$.  This can be verified using the fact that $\FM(S)$ is open in $\F(S)$, showing that the projective quotient map on $\F(S)$ restricts to a quotient map on $\FM(S)$. The openness  follows from the fact that for non-spiral elements $\theta\in \G$, $i_\theta(L)$ is finite in its level, and that this characterizes finite height measured laminations.
\end{remark}

\begin{proof}[Proof of Theorem \ref{SpacesThm}]  The theorem follows from Proposition \ref{InjectiveProp} and Definitions \ref{SpacesDefns}.
\end{proof}

\begin{proof}[Proof of Theorem \ref{GeodesicLamThm}]  As we observed in Corollary \ref{GeodesicInjectiveCor}, restricting the map
$\I:\WM(S)\to [0,\infty]^\G$ described in Lemma \ref{InjectiveLem} to the set $\GL(S)\subset \WM(S)$, we have an injective map $\I:\GL(S)\to [0,\infty]^\G$ with values in $\{0,\infty\}^\G$,  so $\GL(S)$ obtains a topology as a subspace of $\{0,\infty\}^\G$.  Furthermore, $\GL(S)$ becomes a subspace of $\WM(S)\subset [0,\infty]^\G$.  From Definition \ref{WeakDefn}, $\F(S)\subset\WM(S)^{\naturals_0}$ as a subspace, and $\WM(S)$ is contained in the height 0 laminations as $\{(W,\emptyset,\emptyset,\ldots)\in \WM(S)^{\naturals_0}: W\in \WM(S)\}$, so $\F(S)$ contains a homeomorphic copy of $\GL(S)$ as a subspace of the height 0 laminations.

Next we show $\GL(S)$ is totally disconnected. Suppose $L_1,L_2\in \GL(S)$, $L_1\ne L_2$.   Without loss of generality there exists a non-oriented $\theta\in \G$ such that $\theta$ is a leaf of $L_1$ but is not a leaf of $L_2$.  If we can find such a $\theta$ which intersects $L_2$ (transversely), then $i_\theta(L_1)=0$ while $i_\theta(L_2)=\infty$.  Then the sets $U=\{L\in \GL(S): i_\theta(L)=0\}$ and $V=\{L\in \GL(S): i_\theta(L)=\infty\}$ are disjoint, non-empty open sets which exhaust $\GL(S)$, so $L_1$ and $L_2$ cannot be in the same component.  In the remaining case, we can assume that $L_1$ and $L_2$ intersect in a sublamination of each, possibly the empty lamination.  Without loss of generality, there is a leaf $\gamma$ of $L_1$ which is not a leaf of $L_2$ and is in the complement of $L_2$.  If $\hat F$ denotes the complement of $L_2$ and $F$ denotes the completion of $\hat F$, we can regard $\gamma$ as an essential geodesic in $F$, not a component of the boundary.  Then as in the proof of Lemma \ref{InjectiveLem}, we conclude that there is a $\rho\in\G$,  whose underlying geodesic is in $F$, such that $\rho$ intersects $\gamma$.  This means $i_\rho(L_1)=\infty$ and $i_\rho(L_2)=0$.  As before, this shows that $L_1$ and $L_2$ are not in the same component.
\end{proof}

\begin{proof}[Proof of Theorem \ref{DefnsThm}]  The elements of $\F(S)$ can be described in two ways:

(1) The set of essential laminations with $\SB$-measures up to level alignment.  The proof of Lemma \ref{DefnsLemma} shows that every such lamination is equivalent by level alignment uniquely to $L=L_0\cup L_1\cup \cdots \cup L_{h}$ where each $L_k$  is non-empty and weakly measured in $S\setminus (L_{k+1}\cup L_{k+2}\cup \cdots \cup L_{h})$.

(2) The set of essential laminations $L$ uniquely expressible as $L=\bigcup_{i\in \nat} W_i$ where $W_i\supset W_{i+1}$, $W_i$ is weakly measured in $S$, and $L$ satisfies the conditions described in Definition \ref{WeakDefn}.

We have to show that the following two topologies are the same:

(1) The topology induced by the injection $\I=(i_\theta)_{\theta\in\G}:\F(S)\to \SB^\G$.

(2) The topology induced by an injection $\F(S)\to \WM(S)^{\nat}$.

The topology on $\WM(S)$ is induced by the maps $(j_\theta)_{\theta\in \G}:\WM(S)\to [0,\infty]^\G$.  The topology from  $\WM(S)^{\nat}$ is induced on $\F(S)$ by $\J:\F(S)\to \left([0,\infty]^\G\right)^{\nat}$ defined by $\J(L)=\left((j_\theta(W_{i})_{\theta\in \G}\right)_{i\in \nat}$.  Reordering the factors of the product, the topology is induced by the injection $\J(L)=\big((j_\theta(W_{i})_{i\in \nat}\big)_{\theta\in \G}$.  Then we observe that
$\big(j_\theta(W_{i})\big)_{i\in \nat}\in \SBB\subset [0,\infty]^\nat$, so $\J(L)\in \SBB^\G$.  If we use our identification of $\SB$ and $\SBB$, described in Proposition \ref{LHomeoProp}, we now only need to show $\I=\J$ to show the two topologies coincide.  Here $\I:\F(S)\to \SB^\G$ is the function with coordinates $i_\theta:\F(S)\to \SB$.  But this follows from $\psi i_\theta(L)=\big( j_\theta(W_{i})\big )_{i\in\nat}$, where $\psi$ is the isomorphism $\psi:\SB\to \SBB$.
\end{proof}

\begin{proof} [Proof of Corollary \ref{MeasuredRedefineCor}]  This follows immediately from the proof of Theorem \ref{DefnsThm}.
\end{proof}


\begin{proof} [Proof of Proposition \ref{MeasProp}]  The laminations in $p\inverse(K)\subset \F(S)$ are those that have the form $(L_0,K_0,K_1,\ldots ,K_h)$.  We have a map $f:p\inverse(K)\to \WM(F)$, $f(L_0,K_0,K_1,\ldots ,K_h)=L_0$.  To see that this is a continuous map, we think of $(L_0,K_0,K_1,\ldots ,K_h)$ as an element $(W_0,W_1,\ldots, W_{h+1},\emptyset,\ldots)\in (\WM(S)^\nat$.  Projection to the first factor gives $W_0$ which is a weakly measured lamination $W_0\in \WM(S)$ with infinite transverse measure on $K$.
Let $\K\subset \G$ be the set of geodesics in $F$, including oriented variants.  Then the projection $\pi:\bar\reals^\G\to \bar\reals^\K$ is continuous and $\pi((W_0,W_1,\ldots, W_{h+1},\emptyset,\ldots))=L_0$ corresponds to the lamination $L_0\embed F$, $L_0\in \WM(F)$.  This shows $f$ is continuous, a composition of projections in two products.  To construct a continuous inverse function $g:\WM(F)\to p\inverse(K)\subset \F(S)$, we first construct a continuous function $h:\WM(F)\to \WM(S)$ with $h(L)=K\cup L$ with the measure on $L$ unchanged and with infinite transverse measure on $K$.  In terms of coordinates $i_\theta(h(L))=i_\theta(L)$ if $\theta\in \K$, otherwise $i_\theta(h(L))=i_\theta(K)$.   This shows $h$ is continuous.  Now  regard $\F(S)$ as a subspace of $\WM(S)^\nat$ again, and define $g:\WM(F)\to p\inverse(K)$ by $g(L)=(h(L),W_1,W_2,\ldots W_{h+1})$, where  $K=(W_1,W_2,\ldots W_{h+1})\in WM(S)^\nat$.  This shows $g$ is continuous.
\end{proof}

\begin{lemma} \label{DepthBoundLemma} Consider finite height laminations $L=L_0\cup L_1\cup \cdots L_h$ with the property that each $L_i\ne\emptyset$.  For a fixed surface-with-cusps $S$, there is an integer $m$ such that every essential finite height lamination $L$ as above has height $h\le m$.
\end{lemma} 
\begin{proof}  This is an easy exercise involving the geometric Euler characteristic.
\end{proof}

We will say that $m$ is the {\it maximum total height} in $S$.

\begin{lemma}\label{InvariantLemma} Let $S$ be a connected surface (with cusps and boundary curves) satisfying $\chi_g(S)<0$.  Suppose $S$ is not closed and not a disk with boundary cusps.  If $f:S\to S$ is an orientation-preserving automorphism, then there is a non-trivial measured lamination $(L,\mu)$ in the sense of \cite{UO:Pair}, such that $f((L,\mu))=\lambda(L,\mu)$ up to isotopy, with $\lambda>0$. 
\end{lemma}

\begin{proof}  Recall that we use $b$ to denote the number of ``interior cusps," corresponding to closed curves of $\alpha$ in the truncated version $(S,\alpha)$ of $S$.  We use $c$ to denote the number of boundary cusps.  Let us use $d$ to denote the number of closed boundary curves of $S$.  

Suppose first that $c>0$.  This means that some of the boundary components of the underlying truncated surface $\bar S$ contain arcs of $\alpha$.  Consider a curve system $C$ in $\bar S$ consisting of closed curves, each boundary-parallel to a component of the boundary of the underlying surface of $\bar S$ containing arcs of $\alpha$ but not isotopic to a component of $\delta$, the closure of $\bdry S\setminus \alpha$, and with each such isotopy class represented exactly once.  (If the underlying surface of $S$ is an annulus, $C$ could consist of one curve boundary parallel on both sides.)   Then $C$ cuts off a surface whose components are cusped annuli (each with $d=1$ and $c\ge 1$).  These annuli were called {\it trim annuli} in \cite{UO:Pair}.   In the special case that $\bar S$ has underlying surface an annulus, with $d=1$, we can take $C=\emptyset$, and $\bar S$ is itself a trim annulus.  Up to isotopy, $f$ preserves $C$ and it preserves the union of trim annuli.  So we are reduced to the case of automorphisms of disjoint unions of trim annuli and of surfaces with $c=0$.  The latter are surfaces with interior cusps and boundaries, $b+d>0$.

We first deal with automorphisms of surfaces with $c=0$, $b+d>0$.  From Theorem \ref{MeasuredSpaceThm}, we know that if $b>0$, the projective measured lamination space $\PM(S)$ is homeomorphic to a ball of some dimension.  Therefore, by the Brouwer Fixed Point Theorem, the map induced by $f$ on $\PM(S)$ has a fixed point, and the fixed point yields a lamination $(L,\mu)$ such that $f((L,\mu))=(L,\lambda\mu)$ for some $\lambda>0$.   If $b=0$, then $d>0$ and $f$ induces an automorphism $\breve f:\breve S\to \breve S$ of the dual surface $\breve S$.  So we conclude that $\breve f$ has an invariant measured lamination $\breve L$ by the Brouwer Fixed Point Theorem, since $\PM(\breve S)$ is a ball.   The lamination $\breve L$ can be converted to an invariant measured lamination for $f:S\to S$ by replacing each end of a leaf which is mapped to cusps of $\breve S$ by an end which spirals towards the corresponding boundary component of $S$.  We choose the sense of spiraling from the orientation of $S$ using the same convention on all boundary components.  This guarantees that the resulting measured lamination $L$ is invariant for $f$.  

It remains to consider automorphisms of  a trim annulus.  Suppose $S$ is a trim annulus, with $d=1$ and $c>0$.    In this case consider a lamination $L$ consisting of $c$ geodesics, each spiraling to the closed boundary curve and each having its opposite end in a boundary cusp.  The geodesics are chosen so that each boundary cusp contains exactly one end of a geodesic.  We choose a sense of spiraling using the orientation of $S$ according to our uniform convention.  Clearly any orientation preserving automorphism of $S$ preserves $L$.  \end{proof}

\begin{proof}[Proof of Theorem \ref{InvariantThm}]  We are given $f:S\to S$.  In case $S$ is closed the Nielsen-Thurston theory gives either an invariant system of reducing curves (eigenvalue $\lambda=1$), or it gives  a pair of invariant laminations disjoint from $\bdry S$, with eigenvalues $\lambda$ and $1/\lambda$, $\lambda\ne 1$.  Here it is necessary to exclude periodic automorphisms.

If $S$ is not closed, by Lemma \ref{InvariantLemma}, there is an invariant measured lamination $L_0$, provided the underlying surface for $S$ is not a disk.  Even when the underlying surface of $S$ is a disk, it is sometimes possible to find an invariant lamination consisting of isolated geodesics.   The completion of the complement $L_0$ is a new surface with cusps $S_1$, typically not connected, and with no closed components.  Let us say the components are $F_i$.  By the additivity of $\Chi_g$, the sum $\sum_i \Chi_g(F_i)=\Chi_g(S_1)$.  Further, no $\Chi_g(F_i)\ge0$, otherwise $L_0$ would have $\bdry$-parallel leaves, trivial closed curves, or isotopic leaves, which we do not allow in our essential laminations.  Thus $f$ induces an automorphism $g$ of $\cup_i F_i$.  The components $F_i$ are permuted by $g$, so we can divide them into orbits.  Choosing a representative $F_i$ from one orbit, $g^m(F_i)=F_i$ for some $m$.  By the lemma, $g^m$ has an invariant lamination $\hat L$ whose iterates form an invariant measured lamination $L_1$ in $S_1$, provided $F_i$ does not have a contractible underlying surface.  Now we consider the induced automorphism on the completion $S_2$ of $S_1\setminus L_1$, and we obtain another invariant lamination $L_2$.  Continue by induction to obtain invariant measured laminations $L_i$ in $S_i$.  The process stops by Lemma \ref{DepthBoundLemma} when all of the complementary surfaces of $\cup_i L_i$ in $S$ are contractible. 
\end{proof}

\section{Examples}\label{Examples}

At this writing, train tracks are not nearly as useful for analyzing finite height measured laminations in a surface as they are for analyzing measured laminations in a surface.
A measured lamination carried by a train track $\tau$ induces weights on the segments of $\tau$. One uses a weight vector which assigns a non-negative real weight to each segment of the train track such that at each switch of the train track the appropriate {\it switch equation} is satisfied, see Figure \ref{DepthSpiral}.  Weight vectors satisfying the switch equations are called {\it invariant weight vectors}.  An invariant weight vector for a train track $\tau$ determines a measured lamination carried by $\tau$ up to some minor modifications of the lamination.

Similarly a finite height  lamination carried by a train track $\tau$ induces weights in $\SB$ on the segments of the train track, again satisfying the switch equations.  In general, however, these weights do not come close to determining the finite height measured lamination.  Nevertheless, it is useful and instructive to study {\it invariant} $\SB$ weight vectors on train tracks, i.e. weight vectors with entries in $\SB$ which satisfy switch equations.  In simple examples, they correspond exactly to finite height laminations.

\begin{example}\label{SpiralExample}  The simplest (but atypical) examples of a finite height laminations are ones with height 1 leaves consisting of closed curves and height 0 leaves
spiraling to limit on closed curves.  We show an example of such a lamination $L$ in Figure \ref{DepthSpiral}.  We assign a height 1 atomic measure of $(1,1)\in \SB$ to the closed curve, and we assign
an atomic measure of $(0,1)$ to the spiraling leaves.  These measures induce weights on segments of the train track $\tau$ carrying $L$ as shown, and we see that the switch equations
hold with our addition operation for $\SB$.  Ordering the segments, we obtain an invariant weight vector $[x,y,z]=[(0,1),(1,1),(0,1)]\in \SB^3$.  Notice that $w=x+y=(0,1)+(1,1)=(1,1)$ is determined by the other weights.

\begin{figure}[H]
\centering
\scalebox{1}{\includegraphics{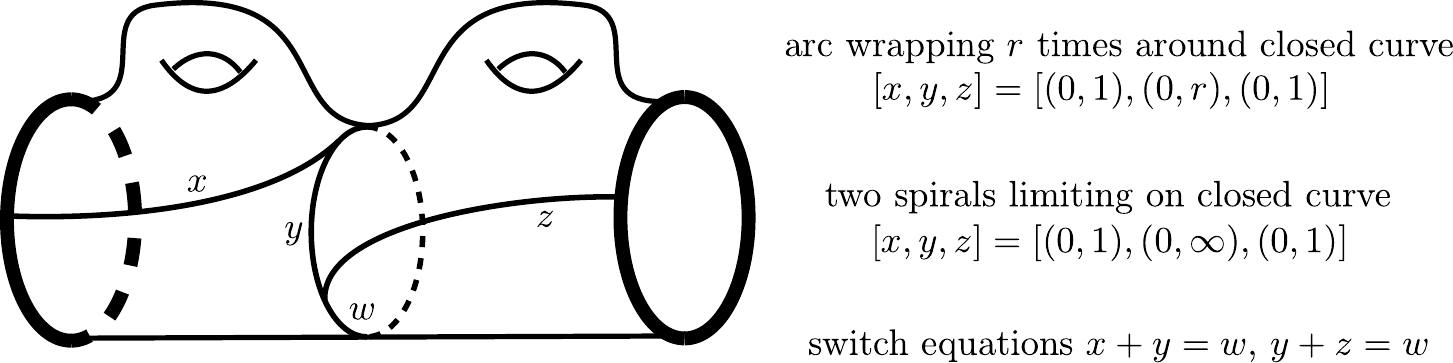}}
\caption{\small Weights in $\SB$ inducing height 0 and height 1 measured laminations.}
\label{DepthSpiral}
\end{figure}

A more interesting feature of this example comes from the fact that the train track also fully carries measured laminations.   Thus for example if all weights have level 0, we have an invariant weight vector $[(0,1),(0,r),(0,1)]$, which represents a curve wrapping $r$ times around the embedded closed curves in $\tau$. Normalizing to make the largest weight $(0,1)$, we obtain a representative $[(0,1/r),(0,1),(0,1/r)]$, which approaches $[0,(0,1),0]$ in $\SB^3$.
If instead we normalize so the smallest weight is $(0,1)$, we are left with \\$[(0,1),(0,r),(0,1)]$, which approaches $[(0,1),(0,\infty),(0,1)]$.  These limit weight vectors clearly represent different finite height laminations.  The weight vector $[0,(0,1),0]$ represents a finite height lamination consisting of a closed curve with weight $(0,1)$.  The weight vector  $[(0,1),(0,\infty),(0,1)]$ represents a finite height lamination consisting of a closed curve with weight $(0,\infty)$ and two spirals approaching this closed curve, each with weight $(0,1)$.   Assuming that the natural topology on projectivized weight vectors correctly represents the topology on $\PF(S)$, we see that on $\PF(S)$ is non-Hausdorff.
\end{example}

We should emphasize again that the above example is not typical, since the weights on the train track determine a finite height lamination.   In general, we will need infinitely many parameters in $\SB$ to determine a finite height
lamination carried by $\tau$.

To get some idea of the nature of the spaces of finite height laminations, we will analyze small portions of the spaces.  Namely, we will let $\F(\tau)$ denote the subspace of $\F(S)$ consisting of all finite height laminations carried by $\tau$.  Similarly for all the other spaces, so that $\PF(\tau)$, $\FC(\tau)$, $\PFC(\tau)$, $\FM(\tau)$, etc., are subspaces of laminations in $\PF(S)$, $\FC(S)$, $\PFC(S)$, $\FM(S)$, etc., respectively.

To analyze finite height measured laminations carried by a train track $\tau$ with $n$ segments it is reasonable to consider the cone of all weight vectors in $\SB^n$ satisfying switch equations.   

\begin{defn} If $\tau \embed (\bar S, \alpha)$ is an essential train track with $n$ segments, then there is a {\it cone $\V(\tau)$ of invariant proximal $\SB$-weight vectors for $\tau$} in $\SB^n$ consisting of weight vectors satisfying switch equations and the further condition of proximality, which we now define.  There is an equivalence relation by {\it level alignments} on the set of invariant weight vectors.  Namely, given a weight vector $[x_1,\ldots, x_n]$, $x_j=(i_j,t_j)$ let $E=\{\level(x_1), \level(x_2),\ldots,\level(x_n)\}\subset \nat$.  If $E$ has the property that $k\notin E$ and $k+1\in E$, then let $y_j=x_j$ if $\level(x_i)=i_j<k$ and let $y_j=(i_{j-1},t_j)$ if $i_j>k$.  (Reduce the level by one beyond the gap.)  If no level alignments are possible, the weight vector is {\it proximal}.

We extend the equivalence relation by {\it projectivizing}.
Two weight vectors $v$ and $w$ in $\V(\tau)$ are {\it projectively equivalent} if $\real(w_j)=\lambda\real (v_j)$  for some $\lambda\in \reals$, $0<\lambda<\infty$, and $\level(v_j)=\level(w_j)$.  $\PV(\tau)$ denotes the resulting projectivization quotient of $\V(\tau)$.  

There is a subspace $\VC(\tau)$ of $\V(\tau)$, the {\it subspace of contiguous weight vectors}.  A {\it contiguous weight vector} is a weight vector for which no alignment or adjustment modification is possible, where ``adjustment" is defined as follows.  Suppose $\hat \tau$ is a sub train track of $\tau$, $\hat \tau$ probably not properly embedded in $(\bar S, \alpha)$.  If $x$ is a switch-respecting weight vector on $\tau$, and we obtain another switch respecting weight vector $x'$ from $x$ by increasing by 1 the levels of all entries representing weights on segments of $\hat \tau$, then $x'$ is obtained from $x$ by an {\it adjustment}.  Thus if $x_i$ is a weight on a segment of $\hat \tau$, $\level(x_i')=\level(x_i)+1$, $\real(x_i')=\real(x_i)$; if $x_i$ is a weight on a segment not in $\hat \tau$, then $x_i'=x_i$.  The space $\PVC(\tau)$ is the space obtained from $\VC(\tau)$ by projectivization.

We can modify the above weight spaces to allow only finite weights at each level.  This gives $\VM(\tau)$, $\PVM(\tau)$, etc., where the ${\mathscr M}$ indicates that we are dealing with finite height measured laminations.
\end{defn}

In words, we can roughly explain level adjustment as follows:  If there is an unnecessary gap in the levels of weights, we eliminate the gap by shifting some of the levels.  
In Example \ref{SpiralExample} obvious modifications are $[(0,1),(2,1),(1,1)]\leadsto[(1,1),(2,1),(1,1)]\leadsto[(0,1),(1,1),(0,1)]$, where the second modification is an alignment.  From this point of view there is no real difference between the finite height laminations represented by these weight vectors.  In the first weight vector the right spiral is deemed to have height 1 and the closed curve is deemed to have height 2, but they could just as well have heights 0 and 1 respectively.   The end result is a contiguous weight vector.

 However, no alignment modification is possible in the weight vector  $[(0,1),(2,1),(1,1)]$.  It is already proximal, so represents a point in $\F(\tau)$.  

In some simple examples we will consider, $\V(\tau)$ and $\F(\tau)$ turn out to be the same, so we will use the names almost interchangeably, but in each example we must verify that the topologies of the two spaces agree.  

Projectivizing in $\SB^n$ is not as obvious as projectivizing $\reals^n$, so we will first describe the projectivization of $\SB^n$ without reference to laminations or train tracks.  When we calculate examples of $\V(\tau)$ or $\PV(\tau)$, many of the features of the projectivized $\SB^n$ appear in our spaces.

\begin{example}[Projectivizing $\SB^2$]\label{ProjectivizeLSquaredEx} 

\begin{figure}[H]
\centering
\scalebox{1}{\includegraphics{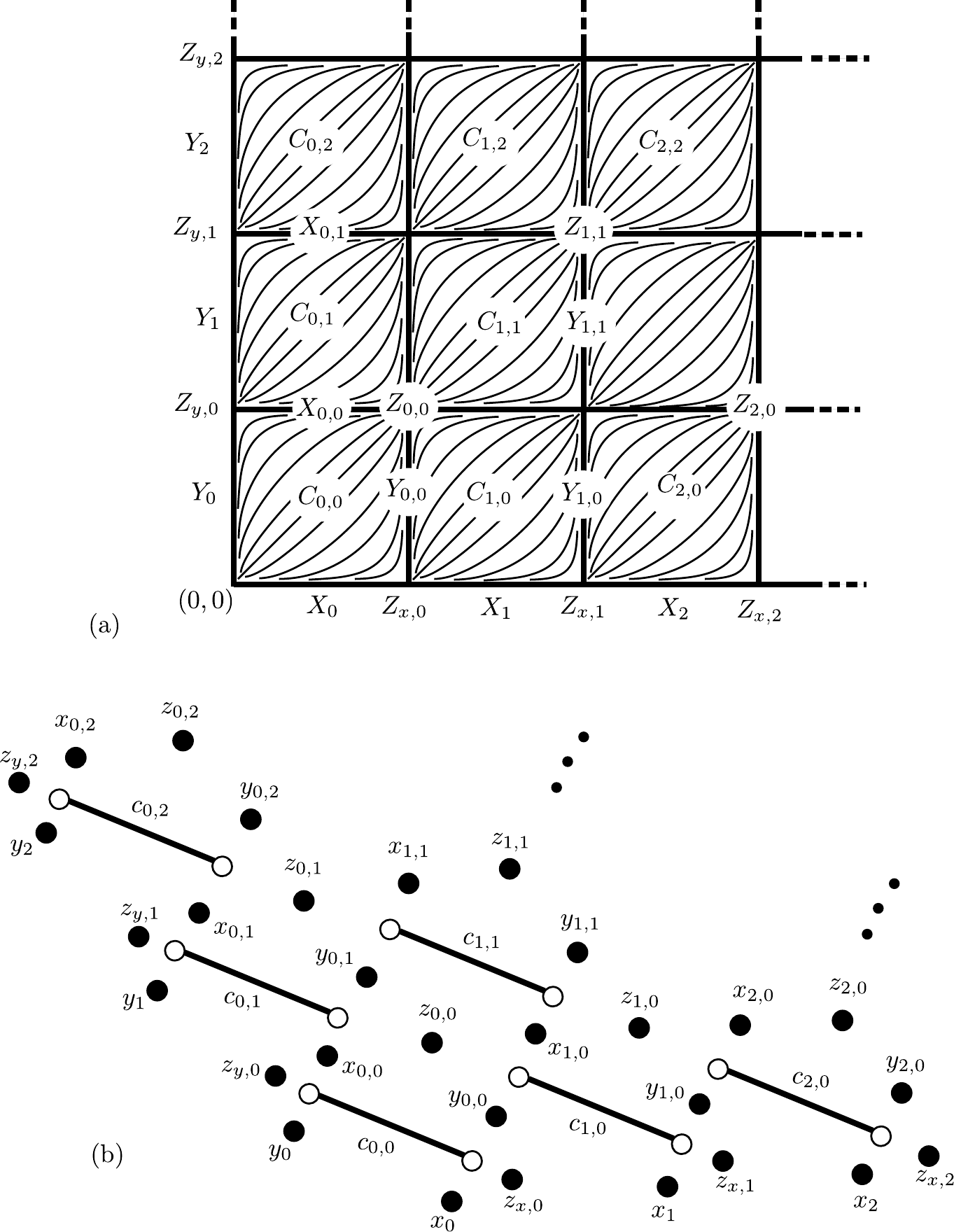}}
\caption{\small Projectivizing $\SB^2$ .}
\label{DepthSSquared}
\end{figure}

We want to find the quotient space $\Proj \SB^2$ of $\SB^2\setminus \{0\}$ obtained from the equivalence relation $v\sim\lambda v$ for $v\in \SB^2$, $\lambda\in (0,\infty)$.    We can imagine $\SB^2$ as a plane tiled by squares, where each square represents a {\it level cone} $C_{i,j}$ having the form $C_{i,j}=\{[x,y]\in \SB^2:\level(x)=i, \level (y)=j\}$.  To each level cone we associate {\it level axes} $X_{i,j}=\{[(i,t),(j,\infty)]\}$ and $Y_{i,j}=\{[(i,\infty),(j,u)]\}$ as shown in Figure \ref{DepthSSquared}(a).  Note that the axes are ``at infinity" in each $C_{i,j}$, rather than ``at 0."  Thus, for example $Y_{0,0}$ is the right edge of $C_{0,0}$ and $X_{0,0}$ is the top edge of $C_{0,0}$ in the diagram.  To each level cone we also associate a lattice point $Z_{i,j}=[(i,\infty),(j,\infty)]$ as shown.    Each level cone should really be imagined as a copy of $(0,\infty]^2$ with projective equivalence classes being rays in this first quadrant of $\bar\reals^2$.  However, we are forced to draw most of the rays as curves, as shown in the figure.  We have further level axes $X_i=\{((i,t),0)\}$ and $Y_i=\{(0,(i,t)\}$, $0<t\le \infty$.   We also have further lattice points $Z_{x,i}=((i,\infty),0)$ and $Z_{y,i}=(0,(i,\infty))$ as shown.  

The level axes are also equivalence classes, but there are some strange phenomena here.  One sees from the figure that one can find a sequence of rays in a level cone approaching two different axes.  For example, the rays in the level cone $C_{1,1}$ containing points $v_n=[x_n,y_n]=[(1,n),(1,1)]$ limit on the union of axes $X_{1,0}$ and $Y_{1,1}$.  This implies that the quotient space is non-Hausdorff, since the sequence $[v_n]$ approaches both the equivalence class $x_{1,0}$ of the open ray in  $X_{1,0}$  and the equivalence class $y_{1,1}$ containing the open ray  $Y_{1,1}$.  The non-Hausdorff behavior in $\Proj\SB^2$ is even worse at the lattice points $z_{i,j}=[Z_{i,j}]$ of the of the projectivization.  Any neighborhood of such a point must contain all the equivalence classes of rays emanating from $Z_{i,j}$.  We have a schematic picture of $\Proj\SB^2$ in Figure \ref{DepthSSquared}(b).  Corresponding to non-axis rays in a level cone like $C_{1,1}$, there is an open segment $c_{1,1}$ in $\Proj\SB^2$.  Each end of the segment approaches two different classes represented by an axis.  In the case of $c_{1,1}$, at the left end of the segment we have points $y_{0,1}$ and  $x_{1,1}$, which are the equivalence classes of the open rays in axes $Y_{0,1}$ and $X_{1,1}$ respectively, which cannot be separated by disjoint neighborhoods.   At the right end we have $x_{1,0}$ and $y_{1,1}$ corresponding to  $[X_{1,0}]$ and $[Y_{1,1}]$ respectively, which also cannot be separated.  Finally, any neighborhood of the point $z_{0,0}$ for example, must contain all points in $c_{0,0}$, $c_{1,1}$ .  The neighborhood must also contain some points of $c_{0,1}$ and $c_{1,0}$ as well as  $y_{0,0}$, $x_{0,0}$, $x_{1,0}$ and $y_{0,1}$.   Similarly, for example, a neighborhood of $z_{y,1}$ must contain some points of $c_{0,1}$, all points of $c_{0,2}$, and $x_{0,1}$.
\end{example}

Now we discuss the simplest possible non-trivial space $\F(\tau)$.

\begin{example}\label{TwoCurves}  Let $S$ be a surface, say a closed surface.  Let $\tau$ be the disjoint union of two disjoint embedded non-isotopic curves $A$ and $B$.  Assigning $\SB$-weights $x$ and $y$ to these curves, we obtain all possible finite height laminations carried by $\tau$.  We are in the situation of the previous example, since there are no constraints on the weights $x$ and $y$, except that for $\V(\tau)$ we restrict to proximal weight vectors and for $\VC(\tau)$ we restrict to contiguous weight vectors.  Since there are two weights, proximal weight vectors can involve weights in at most two levels.   Figure \ref{DepthTwoCurves}(a) shows $\V(\tau)$; (b) shows $\PV(\tau)$; (c) shows $\VC(\tau)$; and (d) shows $\PVC(\tau)$.   Each of these projectivized spaces is a subspace of $\Proj\SB^2$.  Observe that $\PVC(\tau)=\PWM(\tau)$, where $\PWM(\tau)$ denotes the subspace in $\PWM(S)$ of classes of weakly measured laminations in $S$ carried by $\tau$.

 \begin{figure}[H]
\centering
\scalebox{1}{\includegraphics{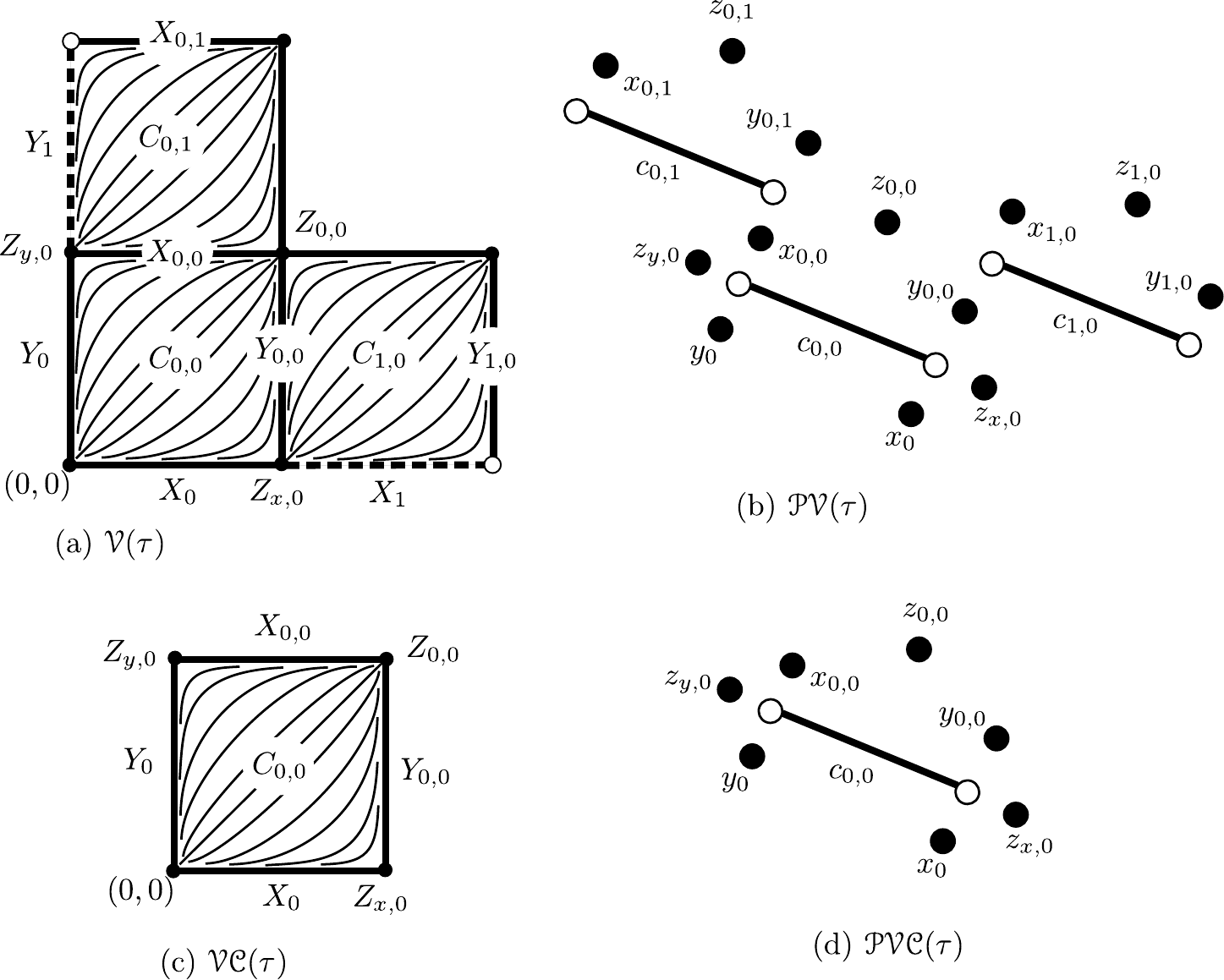}}
\caption{\small Example of $\tau$ a disjoint union of two closed curves.}
\label{DepthTwoCurves}
\end{figure}

To show that $\PV(\tau)$ is homeomorphic to $\PF(\tau)$, we must produce a homeomorphism which takes the projective class of a lamination represented by certain weights on $\tau$ to the same lamination viewed as a point in $\PF(\tau)\subset \SB^\G$.  In this example, this is easy to do, because $\tau$ carries only one geodesic lamination consisting of the two closed curves, which has different finite height structures.  So $\V(\tau)$ can be regarded as $\SB$-linear combinations of $A$ and $B$.  For any $\theta\in \G$, we make the underlying curve $\gamma$ geodesic to see that $i_\theta(xA+yB)=xi_\theta(A)+yi_\theta(B)$, $x,y\in\SB$, where $i_\theta(A)$, for example, represents the possibly infinite counting intersection of $\theta$ with $A$.  This is clearly an $\SB$-linear and continuous.  So $i_\theta$ is continuous on $\V(\tau)$, so the map from $\V(\tau)$ to $\F(\tau)$ is a continuous bijection.  The inverse is also continuous, since $x$ and $y$ can be expressed as a linear combination of finitely many entries of a point in $\F(\tau)\subset\SB^\G$.   Projectivizing both spaces, we also get a homeomorphism $\PV(\tau)\to \PF(\tau)$.
\end{example}

\begin{example} \label{ProjectivizationExample} 
For a slightly more interesting example, we will consider the surface pair $(\bar S, \alpha)$ shown in Figure \ref{DepthCurveSpiral}(a).  We will find $\PF(\tau)$ and other spaces for the train track $\tau$ shown.  There are just two weights, so $\V(\tau)$ is at most ``2-dimensional." 
 
 \begin{figure}[H]
\centering
\scalebox{0.9}{\includegraphics{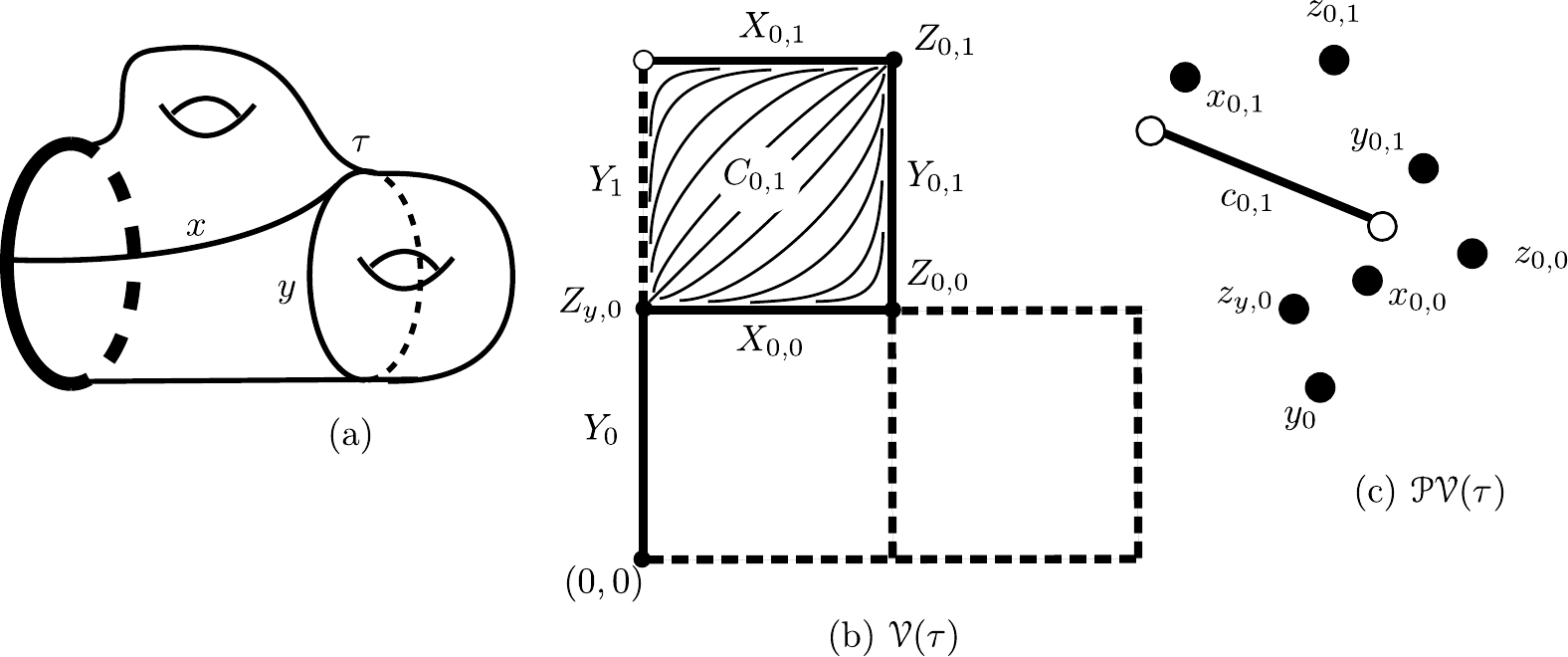}}
\caption{\small Example of $\PV(\tau)$ or $\PF(\tau)$.}
\label{DepthCurveSpiral}
\end{figure}

The space $\V(\tau)$ is again a subspace of $\SB^2$.  We use coordinates $x$ and $y$ in  $\SB^2$ corresponding to weights on the segments of $\tau$.  Again, the train track $\tau$ carries a unique geodesic lamination consisting of a spiral leaf and a closed leaf, and the weights $x$ and $y$ can be thought of as weights on these two leaves.

The only switch equation $x+y=y$, together with the proximality condition, yields a subspace of $\SB^2$.  In order for the equation to hold, $\level(x)<\level(y)$ or $\level(x)=\level(y)$ and $\real(y)=\infty$.  Allowing some 0 weights, we also have $x=0$ and $y$ arbitrary.  If we consider only proximal weight vectors, we obtain the portion of $\SB^2$ shown in Figure \ref{DepthCurveSpiral}(b), which is $\V(\tau)$.
Projectivizing, we obtain $\PV(\tau)$ as show in  Figure \ref{DepthCurveSpiral}(c).  We have the usual non-Hausdorff behavior.   For example $x_{0,0}$ and $y_{0,0}$ cannot be separated by open neighborhoods.  Any neighborhood of $z_{y,0}$ must include all points of $c_{0,1}$, and it must also contain $x_{0,0}$ and $y_0$.

To show that $\PV(\tau)$ is homeomorphic to $\PF(\tau)$ we use essentially the same argument as in the previous example, where now $A$ represents the spiral leaf and $B$ represents the closed leaf of the only lamination fully carried by $\tau$.  

 Here are some interpretations of the points as finite height laminations:  Points in $C_{0,1}$ correspond to $x=(0,t)$ and $y=(1,u)$, $0<u,t<\infty$ are laminations consisting of the closed curve with weight $(1,u)$ and a spiral leaf with weight $(0,t)$.  Points on $X_{0,0}$ correspond to the same lamination with weight $(0,t)$ on the spiral leaf, and weight $(0,\infty)$ on the closed leaf.  The point $Z_{0,0}$ is the same lamination with weight $(0,\infty)$ on both leaves.  The point $Z_{0,1}$ is the same lamination with weight $(0,\infty)$ on the spiral leaf and weight $(1,\infty)$ on the closed leaf; this point does not represent a contiguous lamination, since the level of the spiral can be moved up to $1$.  The point $Z_{0,1}$ and the points on $X_{0,1}$ also do not represent contiguous laminations.   Thus $\PVC(\tau)=\PV(\tau)\setminus\{x_{0,1},z_{0,1}\}$.
\end{example}

We give an example to show that switch-respecting $\SB$ weights on a train track $\tau$ do not determine a finite height measured lamination uniquely.

\begin{example}  Figure \ref{DepthAmbiguous} shows train tracks $\tau_1$ and $\tau_2$ with $\SB$ invariant weights (induced by finite height laminations) and pinching maps $\tau_1\to \tau$ and $\tau_2\to \tau$ inducing the same weights on $\tau$.  Observe that the weights on $\tau$ are not proximal.  To make them proximal, the level of all weights must be reduced by 1.
\end{example}

\begin{figure}[H]
\centering
\scalebox{1}{\includegraphics{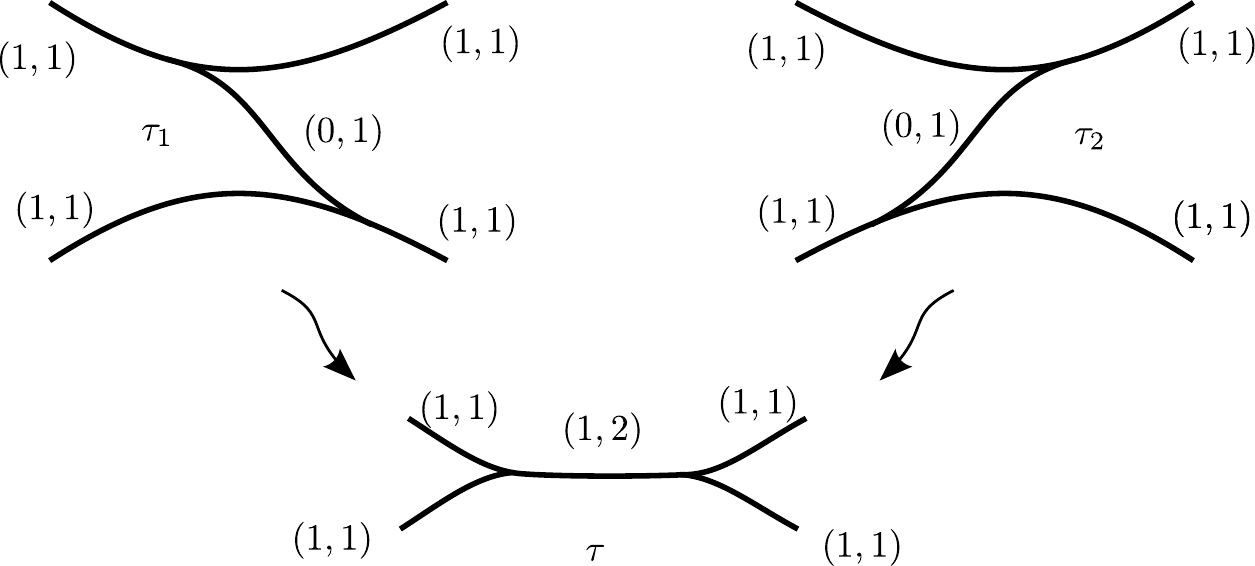}}
\caption{\small Weight vectors on a train tracks do not determine a finite height lamination.}
\label{DepthAmbiguous}
\end{figure}

Based on the example, one might hope that sufficient splitting of a train track would eliminate the ambiguity, so that $\SB$- weights would determine the lamination carried -- but that is hopeless.  However, it seems possible to describe $\F(\tau)$ in terms of inverse limits and infinitely many splittings of the given train track $\tau$.  

\section{Group actions on finite height trees}\label{TreesSection}

Corresponding to an essential lamination $L$ in a closed connected surface $S$, there is an action of $\pi_1(S)$ on an order tree dual to the lift $\tilde L$ of $L$ to the universal cover $\tilde S$ of $S$.  If $L$ is measured, one obtains an action on an $\reals$-tree.  Conversely, from an action of $\pi_1(S)$ on a $\reals$-tree one can construct (non-uniquely) an essential measured lamination.   It is natural to try to describe actions on trees which correspond to finite height measured laminations.  

We begin be generalizing the standard theory of actions of $\pi_1(S)$  on $\reals$ trees only very slightly.  We will describe actions on trees corresponding to measured laminations $(L,\mu)$,  of the kind described in  \cite{UO:Pair}, in a surface with cusps $S$.  Recall that these are measured laminations which are allowed to have leaves spiraling towards closed leaves of $\bdry S$.   Leaves may also approach non-compact components of $\bdry S$ asymptotically, and $\bar L=\bdry S\cup L$ forms a lamination.  Looking at the universal cover $\tilde S$ of $S$ and the lift $\tilde L$ to $\tilde S$, a leaf of $\tilde L$ may approach a component of $\bdry \tilde S$ at one end of that component.  We do not allow ``$\bdry$-parallel" leaves in $L$, which are isotopic to leaves of $\bdry S$, but $\bar L$ contains every boundary component as a leaf.  A closed curve $\beta$ in $\bdry S$ approached by spiral leaf of $ L$ has the property that any transversal of $\bar L$ with one endpoint in $\beta$ has infinite measure.   We
ensured that $\bar L$ has this property at every component of $\bdry S$ by assigning infinite atomic measure to all components of $\bdry S$ not approached by a spiral leaf of $L$.  In fact, when dealing with dual trees, we will replace each component of $\bdry S$ (not approached by a spiral leaf) by a product family of leaves parametrized by $[0,\infty]$.  Henceforth, we will assume that $(\bar L,\bar\nu)$ will have no leaves with atomic measures and will have the property that transversals with one end in $\bdry S$ have infinite measure.
Since we have some transversals with infinite measure, when attempting to construct a dual tree, we are forced to use $\bar \reals$-trees rather than $\reals$-trees:

\begin{defns}  If $\bar \reals$ denotes the extended real line, $\bar\reals=[-\infty,\infty]$, an {\it $\bar\reals$ metric space} is a pair $(X,d)$ where $d:X\times X\to \bar\reals$ satisfies the usual metric space axioms.  In an $\bar\reals$ metric space, a {\it segment} is a subset isometric to some $[a,b]\subset \bar\reals$.  An {\it $\bar\reals$-tree} is an $\bar\reals$ metric space $\T$ satisfying the following axioms:

(1)  For any two points $v,w\in \T$ there is a unique segment $[v,w]$  with endpoints $v,w$.

(2) The intersection of two segments $[x,y]$ and $[x,z]$ in $\T$ with a common endpoint $x$ is a segment $[x,w]$.

(3) $[x,y]\cup [y,z]=[x,z]$ if $[x,y]\cap [y,z]=\{y\}$.

The {\it directions} of $\T$ at a point $x$ are the components of $\T-\{x\}$.  A point $x\in \T$ is called a {\it boundary point} if it has only one direction.

An {\it infinite point} in $\T$ is a point $y$ such that there is $x\ne y$ such that $[x,y]$ is isometric to $[0,\infty]$.
\end{defns}


A $\bar\reals$-metric described above is sometimes called an {\it extended metric}. 

It is easy to check that an infinite point is a boundary point.  For if $y$ were an infinite point and not a boundary point then $\T\setminus\{y\}$ would have at least two components, and there would exist $x\in \T$ such that $[x,y]$ is isometric to $[0,\infty]$.  Let $w\in \T$ be in a component of $\T\setminus\{y\}$ not containing $x$.  Then $[w,y]\cap[x,y]=\{y\}$, so $[w,x]=[w,y]\cup [y,x]$ is a segment which is clearly not isometric to an interval in $\bar \reals$.

One can check that if $y$ is an infinite point, then for any $w\ne y$, $[w,y]$ is isometric to $[0,\infty]$ or $[-\infty,\infty]$.  This follows since by definition there exists $x$ so that $[x,y]$ is isometric to $[0,\infty]$.  By the axioms, $[x,y]\cap [w,y]$ is a segment $[r,y]$ which is isometric to $[0,\infty]$ unless $r=y$, but if $r=y$, then $y$ is not a boundary point.  We conclude $[w,y]$ is isometric to $[0,\infty]$ or $[-\infty,\infty]$. 

We have proved part of the following proposition:

\begin{proposition}\label{MeasuredTreeProp} Suppose $(L,\nu)$ is an essential $\reals$- measured lamination in $S$ and $(\bar L,\bar\nu)$ is constructed as above to have the property that any transversal with one end on $\bdry S$ has infinite measure.  If $\T$ is the $\bar\reals$-tree dual to the lift $\tilde{\bar L}$, then there is an action of $\pi_1(S)$ on $\T$.  The following are equivalent:

\begin{tightenum}
\item $q\in \T$ is a boundary point,
\item $q\in \T$ is an infinite point,
\item $q\in \T$ represents a component of $\bdry \tilde S$.
\end{tightenum}
\end{proposition}

\begin{remark} The proposition illustrates the fact that the $\bar\reals$-tree approximates $\breve S$ with a hyperbolic  structure, rather than $S$.
\end{remark}

\begin{proof}  We construct the dual tree in the usual way.  It's points are the components of of $\tilde S\setminus \tilde{\bar L}$ in the universal cover $\tilde S$ and leaves not in the closures of these complementary components. The $\bar \reals $ distance between two points is given by the length of an efficient transversal joining the corresponding leaves or complementary components.  We constructed $(\bar L,\bar\nu)$ to ensure that the distance in $\T$ from a leaf in $\bdry \tilde S$ to some distinct point is always infinite; this corresponds to the property that a non-trivial transversal in $S$ with one end in $\bdry S$ always has infinite measure.  Further,  assuming that $L$ is $\reals$-measured means that transversals in the interior of $S$ have finite measure, so using the discussion preceding the statement of the proposition we conclude that the a point in $\T$ is a boundary point if and only if it is an infinite point if and only if it represents a leaf of $\bdry S$.   It is also clear that if $q\in T$ does not represent a component of $\bdry \tilde S$, then there are at least two emanating directions, corresponding to at least two distinct emanating transversals (not isotopic through transversals with ends in the same leaf or complementary component). In the usual way, we obtain an action of $\pi_1(S)$ on $\T$.  
\end{proof}

\begin{defn} A {\it locally finite $\bar\reals$-tree} $\T$ is a tree having the property that a point $q\in \T$ is a boundary point if and only if it is an infinite point.
\end{defn}

Thus if $L$ is a measured lamination (which has a locally finite transverse measure), the $\bar\reals$- tree $\T$ dual to the lift of $\bar L$ is locally finite.  

Now suppose $L$ is a weakly measured essential lamination in $S$.   There is a sublamination $L_\infty$ which supports the infinite part of the transverse measure.   As before, we must avoid atomic measures on leaves of $L_\infty$, and there is potentially a technical problem here.  We cannot simply replace each leaf by a product family of leaves with infinite transverse measure since this might ``take too much space in the surface."  However, from Proposition \ref{PlanteProp}, we know that $L_\infty$ admits some structure as a finite height measured lamination, $\displaystyle L_\infty=\bigcup_{j=0}^hX_j$, with transverse measure $\nu_j$ on $X_j$.   We construct an infinite transverse measure on $\L_\infty$ without atomic measures as follows.  If some $X_j$ contains isolated leaves (in $\hat S_j$) with atomic measures, we replace these by a product family, with non-atomic transverse measure.  There can only be finitely many such leaves.  Then we define the infinite transverse measure on $L$ as $\sum \infty\nu_j$, multiplying each measure $\nu_j$ by $\infty$.  Finally, to construct $\bar L$ we add parallel families of leaves isotopic to components of $\bdry S$, as before, to ensure that $\bar L$ has the property that any transversal with one end in $\bdry S$ has infinite measure.  The resulting $(\bar L,\bar\nu)$ is again weakly measured, but now the measure need not be locally finite in the interior of $S$.  We say the lamination $(\bar L,\bar\nu)$ is {\it geometric weakly measured.}

\begin{proposition}  Suppose $(L,\nu)$ is a weakly measured geodesic lamination in $S$, and $(\bar L,\bar \nu)$ is the modified geometric lamination.   If $\T$ is the $\bar\reals$-tree dual to the lift $\tilde{\bar L}$, then there is an action of $\pi_1(S)$ on $\T$.  Further, 
\begin{tightenum}
\item $q\in \T$ is a boundary point if and only if $q$ represents a component of $\bdry S$, and
\item  if $q\in \T$ represents a component of $\bdry S$, then $q$ is an infinite point.
\end{tightenum}
\end{proposition}

\begin{proof}  The proof uses the same ideas as the proof of Proposition \ref{MeasuredTreeProp}.
\end{proof}

The trees dual to lifts to the universal cover of arbitrary essential laminations are more general trees called ``order trees."  There is a definition in
 \cite{DGUO:EssentialLaminations}, but we give a different definition.  (I am not sure who first defined these; I first heard about  order trees from Peter Shalen.)
 
 \begin{defn}  An {\it order tree} is a set $\T$ together with a subset $[x,y]$, called a {\it segment}, associated to each pair of elements, together with a linear order on $[x,y]$ such that $x$ is the least element in $[x,y]$ and $y$ is the greatest element.  We allow {\it trivial segments} $[x,x]$.  The set of segments should satisfy the following axioms:
 
 \begin{tightenum}
 \item The segment $[y,x]$ is the segment $[x,y]$ with the opposite order.
\item The intersection of segments $[x,y]$ and $[x,z]$ is a segment $[x,w]$.  
 
\item  If two segments intersect at a single point, $[x,y]\cap[y,z]=\{y\}$ then the union is a segment  $[x,z]$.
 \end{tightenum}
 The order tree is a topological space:  $G$ is open in $\T$ if for every segment $[x,y]$, $G\cap [x,y]$ is open in the order topology for $[x,y]$.
 \end{defn}

Clearly $\reals$-trees and $\bar\reals$-trees are also order trees, as are $\Lambda$-trees, where $\Lambda$ is an ordered abelian group.

Let $\T$ be an order tree and $x\in \T$.  Let $\D_x$ be the set of directions at $x$.  Then for every direction $d$ at $x$ we have a subtree $\T_d$ called the {\it branch in the direction $d$}, the union of segments from $x$ to points $p$ in the component of $\T\setminus \{x\}$ representing the branch.   Now suppose we have another tree $R$ with $\B$ a subset of the boundary points, and suppose we have a one-one correspondence $\psi:\B\to \D_x$.  Then we can replace $x\in \T$ by $R$ to obtain a new tree $\T'=(\cup_d\T_d\cup R)/\sim$ by attaching each point $b\in \B\subset R$ to $x\in \T_d$ where $d=\psi(b)$. 

\begin{defn}  In the above construction,  we say $\T'$ is obtain from $\T$ by {\it inserting} $R$ at $x$ according to the correspondence $\psi$.
\end{defn}

Loosely speaking, to insert $R$ at $x\in \T$, we cut $\T$ at $x$, then attach the branches of $\T$ at $x$ to the boundary points of $R$ according to the correspondence $\psi$.   Note that the topological quotient $\T'/R$ is the tree $\T$.

\begin{defn}  Suppose $\T$ is an order tree and the group $G$ acts on $\T$.  If $X$ is the $G$-orbit of $x_0$ in $\T$, we will describe a simultaneous insertion of isometric trees at all points of $X$, called an {\it equivariant insertion}, which yields a new tree $\T'$ with an extended action of $G$ on $\T'$.  The inserted trees are all isomorphic to a given tree $R$, and we are given an action of the stabilizer $H$ of $x_0$ on $R$ which must satisfy the following condition.  If $h\in H$, then $h$ induces a bijection $\psi$ of directions at $x_0$, $\psi_h:\D_{x_0}\to \D_{x_0}$.  Similarly, if $h\in H$, $h$ induces a bijection of boundary points of $R$ and we assume $\B$ is an invariant subset of the boundary points with induced bijection $b\leadsto hb$, $b\in \B$.  Further we assume we are given a glueing map $\phi:\B\to \D_{x_0}$ which satisfies the condition that for every $h$, $\phi (h b)=\psi_h\phi(b)$

  The action of $G$ on $\T$ induces bijections between elements of $\{\D_x:x\in X\}$.  Namely, there is a bijection $\psi_g:\D_{x_0}\to \D_{g x_0}$ such that $\psi_{gh}=\psi_g\psi_h$.  If $R$ is a tree with $\B$ a subset of the boundary points and $\phi:\B\to \D_{x_0}$ is the bijection for inserting $R$ at $x_0$, then the equivariant insertion also inserts $R$ at other points of $X$ according to the rule that the attaching map for $R$ at $gx_0$ is $\psi_g\phi:\B\to \D_{gx_0}$.  If $H$ is the stabilizer of $x_0$ in $G$, we choose an action of $H$ on $R$ to extend the action to the $R$ inserted at $x_0$.  Then $gHg\inverse $ is the stabilizer of $gx_0$ and we choose the action on the $R$ inserted at $gx_0$ to satisfy $ghg\inverse (gr)=ghr$ for $r\in R$.

If a number of different equivariant insertions are done simultaneously on disjoint orbits to obtain $\T'$ from $\T$, then we still say that $\T'$ is obtained from $\T$ by {\it equivariant insertion}.  If $\T'$ with a $G$-action is obtained from $\T$ with a $G$-action by a finite sequence of equivariant insertions, such that later insertions can be done at points in previously inserted subtrees, then we say $\T'$ with its action is {\it finite height}.  Or we say that the action of $G$ on $\T'$ is a { \it finite height action} on the {\it finite height tree $\T'$}.  

If all the trees involved are $\bar\reals$-trees, and all the actions on the original tree and inserted trees preserve distance, then $\T'$ is a {\it finite height $\bar \reals$-tree}.  If all of inserted trees have the property that points are infinite if and only if they are boundary points, then we say $\T'$ is a  {\it finite height locally finite $\bar\reals$-tree}.
\end{defn}

\begin{defn}  Suppose $(L,\nu)$ is a geodesic finite height lamination, with $\nu$ an invariant transverse $\SB$-measure. If $L=L_0\cup L_1\cup\cdots\cup L_h$, with $L_i$ having an $\bar\reals$-measure $\nu_i$, then  we construct the {\it associated geometric finite height lamination $\bar L=\bar L_0\cup \bar L_1\cup\cdots\cup \bar L_h$}, where the $\bar L_i$'s are the weakly measured laminations constructed from $L_i$'s as described above, with associated transverse measures $\bar\nu_i$.  The transverse $\SB$-measure induced on $\bar L$ is called $\bar\nu$.  
\end{defn}


\begin{proposition} If $S$ is a surface with cusps and boundaries, and $L\embed S$ is a finite height essential lamination, then there is a finite height action of  $\pi_1(S)$ on the finite height $\bar\reals$-tree $\T$ dual to the lift $\tilde{ \bar L}$ of the associated finite height lamination $\bar L$ to the universal cover of $S$.   Similarly, if $L$ is a finite height measured lamination, there is an action of $\pi_1(S)$ on a finite height $\bar\reals$ locally finite tree $\T$ dual to $\tilde {\bar L}$.
\end{proposition}

\begin{proof}  Let $\tilde{\bar L }$ denote the lift of $ \bar L$ to the universal cover of $S$.  We can describe the finite height lamination as $\bar L=\bar L_0\cup \bar L_1\cup\cdots \bar L_h$, where each $\bar L_i$ is weakly measured in the appropriate subsurface with measure $\bar\nu_i$. The theorem is proved by observing that efficient transversals for $\tilde{\bar  L}$  project to appropriate segments in the tree $\T$ dual to $\bar{\tilde L}$.  The tree $\T_h$ dual to $\tilde{\bar L}_h$ is the original tree.  Inductively, corresponding lifts $\tilde {\bar L}_i$, we have dual trees to be inserted equivariantly in the previous tree $\T_{i+1}$, and there is an obvious action of $\pi_1(S)$ on the final tree obtained.    If all the laminations $L_i\subset S_i$ are measured, then we can inductively equivariantly insert locally finite $\bar\reals$-trees and define a suitable action on the resulting finite height tree.  \end{proof}

A suitable converse, to produce (non-uniquely) a finite height essential measured lamination from a finite height action can probably be stated and proved, but there must be conditions relating point stabilizers to subgroups of $\pi_1((S,\alpha))$  corresponding to fundamental groups of suitable subsurface pairs.

We now present another point of view to describe actions on trees dual to finite height laminations.

\begin{defn} An $\SB$-metric on a set $X$ is a function $d:X\times X\to \SB$ satisfying the usual axioms for a metric.  An $\SB$-metric space is the set $X$ together and $\SB$-metric.
\end{defn}

\begin{defn}  Suppose $\T$ is an order tree and suppose $\nu$ is an $\SB$-measure $\nu$ on the disjoint union of segments of $\T$ with the property that  if $[x,y]$ and $[z,w]$ are segments, and $[x,y]\cap [z,w]=[u,v]$, then for any measurable set $E\subset [u,v]$, $\nu(E)$ is the same no matter which segment ($[x,y]$,$ [z,w]$, or $[u,v]$) we use to evaluate the measure.  (The measure agrees on intersections of segments.)   We say $\nu$ is an {\it $\SB$-measure on $\T$} and $\T$ with the measure  $\nu$ is called an {\it $\SB$-tree}.

The $\SB$-measure on an order tree is {\it non-atomic}  if the measure of a single point in a segment is always $0$.  It has {\it full support} if it has full support on the disjoint union of segments.  If the $\SB$-measure associated to an $\SB$-tree $\T$ has full support and is non-atomic, then we say the $\SB$-tree is a {\it metric $\SB$-tree}.
\end{defn}

\begin{lemma}  Suppose $\T$ is a metric $\SB$-tree, meaning the associated $\SB$-measure has full support and is not atomic at any point.  Then $\T$ is an $\SB$-metric space with metric $d(x,y)=\nu([x,y])\in \SB$.
\end{lemma}

\begin{proof}  Because $\nu$ has no atomic measures on points, we conclude $d(x,y)=\nu([x,y])=0$ if and only if $x=y$.  To verify the triangle inequality, observe that if $x,y,z$ are points in the tree, by axiom (iii) for order trees, $[x,y]\cap [x,z]=[x,w]$ for some $w$, so $[y,w]\cup [w,z]=[y,z]$ by axiom (iii).  Hence $d(y,z)=\nu ([y,z])=\nu([y,w])+\nu([w,z])\le \nu([y,x])+\nu([x,z])=d(y,x)+d(x,z)$, because $[y,w]\subset [y,x]$ and $[w,z]\subset [x,z]$.
\end{proof}

\begin{defn}  Suppose $L$ is an essential  lamination in $S$   The {\it order tree dual to the lift $\tilde L$ of $L$ to the universal cover $\tilde S$ of $S$} is the set of closures of complementary regions of $\tilde L$ union non-boundary leaves.  A {\it segment} $[x,y]$ is the set elements of $\T$ intersected by closed oriented efficient transversal $T$ for $\tilde L$ with order coming from the order on the transversal.  \end{defn}

\begin{proposition}  Given an essential lamination $L$ in $S$, the object $\T$ defined above with the given segments is an order tree.
If $L$ is $\SB$-measured, with measure $\mu$, then the lifted measure $\nu=\tilde\mu$ yields an $\SB$-measure $\nu$ for $\T$, so $\T$ is and $\SB$-tree..  If $\mu$ has no leaves with atomic transverse measures, $\T$ is an $\SB$-metric space with metric $d(x,y)=\nu([x,y])$ for $x,y\in \T$.
\end{proposition}

\begin{proof}  We verify the order tree axioms:  (i) is true by construction, $[y,x]$ is $[x,y]$ with the opposite order, coming from a transversal with the opposite orientation.  

For (ii) we must show that $[x,y]\cap [x,z]=[x,w]$ for some $w\in \T$.  Consider oriented geodesic segments $\gamma$ from a point in $X$ to a point in $Y$, and $\beta$ from a point in $X$ to a point in $Z$.   We make the convention that the point in the tree corresponding to a leaf or complementary component $X$ will be called $x$, changing to lower case.   Choose a $v>x$, $v\in [x,y]$ such that both $\gamma$ and $\beta$ intersect the corresponding $V$.  Choose a geodesic segment $\omega$ in $V$ joining a point in $\gamma\cap V$ to a point in $\beta\cap V$, and similarly choose a geodesic segment $\rho$ joining a point $\gamma\cap X$ to a point in $\beta\cap X$.  Without loss of generality, $\rho$ and $\omega$ intersect $\beta$ and $\gamma$ at endpoints of  $\beta$ and $\gamma$.  Consider the rectangular disk $\Delta$ bounded by $\gamma$, $\beta$, $\rho$ and $\omega$.  Geodesic segments of $\tilde L\cap \Delta$ do not intersect $\omega$ or $\rho$, therefore every such geodesic which intersects $\gamma$ must also intersect $\beta$, and vice versa.  This proves that $[x,v]$ viewed as a subset of $[x,y]$ is the same as  $[x,v]$ viewed as a subset of $[x,z]$.  We must now show there is a maximal $v$ of this kind.  We parametrize $\gamma:[0,1]\to \tilde S$, and let $u=\sup\{t:\gamma(t)\in V, \text{ such that also } v\in[x,z]\}$ and let $W$ contain $\gamma(u)$, with $w$ the corresponding point in the tree.  Now build a rectangle as before, with $W$ instead of $V$: Choose a geodesic segment $\omega$ in $W$ joining a point in $\gamma\cap W$ to a point in $\beta\cap W$, and similarly choose a geodesic segment $\rho$ joining a point $\gamma\cap X$ to a point in $\beta\cap X$.  Without loss of generality, $\rho$ and $\omega$ intersect $\beta$ and $\gamma$ at endpoints of  $\beta$ and $\gamma$.  Consider the rectangular disk $\Delta$ bounded by $\gamma$, $\beta$, $\rho$ and $\omega$.  Again geodesic segments of $\tilde L\cap \Delta$ do not intersect $\omega$ or $\rho$, and by what we have proved already, every such segment in $\Delta\setminus W$ joins a point in $\beta$ to a point in $\gamma$.  The same is then also true for a segment of $\bdry W\cap \Delta$, which shows that $[x,y]\cap [x,z]=[x,w]$.

For property (iii), suppose $[x,y]$ and $[y,z]$ are (non-trivial) segments in $\T$ with $[x,y]\cap [y,z]=\{y\}$.    Representing $[x,y]$ by an oriented geodesic segment $\beta$ and $[yz]$ by an oriented geodesic segment $\gamma$, $\beta\cup \gamma$ must be an embedded path, otherwise the two segments would coincide on a non-trivial subsegment, a contradiction.   It follows that $\beta\cup \gamma$ can be regarded as a transversal, representing $[x,z]$.  

Now that we know that $\T$ is an order tree, it is easy to show it is an $\SB$-tree.  The transverse $\SB$-measure $\mu$ for  $L$ yields a transverse measure $\tilde \mu$ for $\tilde L$, which in turn gives a measure on transversals.  Since transversals are identified with segments of $\T$, we have an $\SB$-measure $\nu$ on the transversals.  Invariance of the measure $\tilde \mu$ gives an $\SB$-measure $\nu$ on the disjoint union of segments of $\T$.  If there are no leaves of $L$ with atomic measure, there are no points with atomic measure in (the segments of) $\T$, which shows that $d(x,y)=\nu([x,y])$ defines an $\SB$-metric on $\T$.  
\end{proof}

To obtain a finite height action on a finite height tree according to our previous definition, we must put further conditions on the $\SB$-measure to ensure that the lamination at each level is geometric.  Clearly, if $L$ is a finite height lamination, and we make modifications to obtain the geometric $(\bar L,\bar \nu)$, then the dual tree will be a metric space, so Theorem \ref{TreeThm} follows from the above proposition.

\bibliographystyle{amsplain}
\bibliography{ReferencesUO3}
\end{document}